\documentclass[11pt]{amsart}
\usepackage{amsthm,amsfonts,amsmath,graphicx,calc}
\usepackage[numbers,sort&compress]{natbib}
\usepackage[lmargin=29mm,rmargin=29mm,tmargin=29mm,bmargin=29mm]{geometry}

\renewcommand{\baselinestretch}{1.175}
\usepackage{amsthm}

\theoremstyle{plain}
\newtheorem{theorem}{Theorem}[section]
\newtheorem{lemma}[theorem]{Lemma}
\newtheorem{corollary}[theorem]{Corollary}
\newtheorem{proposition}[theorem]{Proposition}
\newtheorem{observation}[theorem]{Observation}

\theoremstyle{definition}

\newtheorem{open}[theorem]{Open Problem}

\newcommand{\thmlabel}[1]{\label{thm:#1}}
\newcommand{\thmref}[1]{Theorem~\ref{thm:#1}}
\newcommand{\twothmref}[2]{Theorems~\ref{thm:#1} and \ref{thm:#2}}
\newcommand{\lemlabel}[1]{\label{lem:#1}}
\newcommand{\lemref}[1]{Lemma~\ref{lem:#1}}
\newcommand{\twolemref}[2]{Lemmas~\ref{lem:#1} and \ref{lem:#2}}
\newcommand{\eqnlabel}[1]{\label{eqn:#1}}
\newcommand{\eqnref}[1]{\eqref{eqn:#1}}

\newcommand{\figlabel}[1]{\label{fig:#1}}
\newcommand{\figref}[1]{Figure~\ref{fig:#1}}
\newcommand{\seclabel}[1]{\label{sec:#1}}
\newcommand{\secref}[1]{Section~\ref{sec:#1}}
\newcommand{\corlabel}[1]{\label{cor:#1}}
\newcommand{\corref}[1]{Corollary~\ref{cor:#1}}
\newcommand{\proplabel}[1]{\label{prop:#1}}
\newcommand{\propref}[1]{Proposition~\ref{prop:#1}}
\newcommand{\obslabel}[1]{\label{obs:#1}}
\newcommand{\obsref}[1]{Observation~\ref{obs:#1}}
\newcommand{\applabel}[1]{\label{app:#1}}
\newcommand{\appref}[1]{Appendix~\ref{app:#1}}

\newcommand{\Figure}[4][htb]{
\begin{figure}[#1]
	\vspace*{1ex}
	\begin{center}#3\end{center}
	\vspace*{-1ex}
	\caption{\figlabel{#2}#4}
\end{figure}}

\newcommand{\remove}[1]{}

\newcommand{\half}{\ensuremath{\protect\tfrac{1}{2}}}
\newcommand{\bracket}[1]{\ensuremath{\protect\left(#1\right)}}

\newcommand{\N}{\ensuremath{\mathbb{N}}}
\newcommand{\Z}{\ensuremath{\mathbb{Z}}}

\newcommand{\tw}[1]{\ensuremath{\textup{\textsf{tw}}(#1)}}
\newcommand{\pw}[1]{\ensuremath{\textup{\textsf{pw}}(#1)}}

\newcommand{\X}{\ensuremath{\mathcal{X}}}
\newcommand{\XX}{\ensuremath{\widehat{\mathcal{X}}}}

\newcommand{\PP}{\ensuremath{\mathcal{P}}}
\newcommand{\PPP}{\ensuremath{\widehat{\mathcal{P}}}}

\newcommand{\down}[1]{\ensuremath{\widehat{#1}}}
\newcommand{\up}[1]{\ensuremath{\widecheck{#1}}}
\renewcommand{\up}[1]{\ensuremath{\overline{#1}}}

\newcommand{\D}[1]{\ensuremath{\mathcal{D}_{#1}}}
\newcommand{\DD}[1]{\ensuremath{\widehat{\mathcal{D}}_{#1}}}

\newcommand{\C}[1]{\ensuremath{\mathcal{C}_{#1}}}
\newcommand{\CC}[1]{\ensuremath{\widehat{\mathcal{C}}_{#1}}}

\newcommand{\T}[1]{\ensuremath{\mathcal{T}_{#1}}}
\newcommand{\TT}[1]{\ensuremath{\widehat{\mathcal{T}}_{#1}}}

\newcommand{\W}[1]{\ensuremath{\mathcal{P}_{#1}}}
\newcommand{\WW}[1]{\ensuremath{\widehat{\mathcal{P}}_{#1}}}

\newcommand{\B}{\ensuremath{\mathcal{B}}}

\newcommand{\vl}{\ensuremath{V_{\ell}}}
\newcommand{\vh}{\ensuremath{V_{h}}}
\newcommand{\dist}{\mathop{{\rm dist}}}

\begin{document}

\title{Graph Minors and Minimum Degree}

\author{Ga\v{s}per Fijav\v{z}}
\address{\newline Faculty of Computer and Information Science\newline University of Ljubljana\newline Ljubljana, Slovenia}
\email{gasper.fijavz@fri.uni-lj.si}

\author{David R.~Wood}
\address{\newline  Department of Mathematics and Statistics\newline The University of Melbourne\newline Melbourne, Australia}
\email{woodd@unimelb.edu.au}

\date{\today}

\begin{abstract}
Let $\mathcal{D}_k$ be the class of graphs for which every minor has minimum degree at most $k$. Then $\mathcal{D}_k$ is closed under taking minors. By the Robertson-Seymour graph minor theorem, $\mathcal{D}_k$ is characterised by a finite family of minor-minimal forbidden graphs, which we denote by $\widehat{\mathcal{D}}_k$. This paper discusses $\widehat{\mathcal{D}}_k$ and related topics. We obtain four main results:
\begin{enumerate}
\item We prove that every $(k+1)$-regular graph with less than $\frac{4}{3}(k+2)$ vertices is in $\widehat{\mathcal{D}}_k$, and this bound is best possible.
\item We characterise the graphs in $\widehat{\mathcal{D}}_{k+1}$ that can be obtained from a graph in $\widehat{\mathcal{D}}_k$ by adding one new vertex.
\item  For $k\leq 3$ every graph in $\widehat{\mathcal{D}}_k$ is $(k+1)$-connected, but for large $k$, we exhibit graphs in $\widehat{\mathcal{D}}_k$ with connectivity $1$. In fact, we construct graphs in  $\mathcal{D}_k$ with arbitrary block structure.
\item We characterise the complete multipartite graphs in $\widehat{\mathcal{D}}_k$, and prove analogous characterisations with minimum degree replaced by connectivity, treewidth, or pathwidth.
\end{enumerate}
\end{abstract}

\maketitle

\tableofcontents

\newpage
\section{Introduction}



The theory of graph minors developed by \citet{RS-GraphMinors} is one of the most important in graph theory influencing many branches of mathematics. Let \X\ be a minor-closed class of graphs\footnote{All graphs considered in this paper are undirected, simple, and finite. 

To \emph{contract} an edge $vw$ in a graph $G$ means to delete $vw$, identify $v$ and $w$, and replace any parallel edges by a single edge. The contracted graph is denoted by $G/vw$. If $S\subseteq E(G)$ then $G/S$ is the graph obtained from $G$ by contracting each edge in $S$ (while edges in $S$ remain in $G$). The graph $G/S$ is called a \emph{contraction minor} of $G$.

A graph $H$ is a \emph{minor} of a graph $G$ if a graph isomorphic to $H$ can be obtained from a subgraph of $G$ by contracting edges. That is, $H$ can be obtained from $G$ by a sequence of edge contractions, edge deletions, or vertex deletions. For each vertex $v$ of $H$, the set of vertices of $G$ that are contracted into $v$ is called a \emph{branch set} of $H$. A class \X\ of graphs is \emph{minor-closed} if every minor of every graph in \X\ is also in \X, and some graph is not in \X.

The \emph{join} of graphs $G$ and $H$, denoted by $G*H$, is the graph obtained by adding all possible edges between disjoint copies of $G$ and $H$. Let $\overline{G}$ denote the complement of a graph $G$.}. A graph $G$ is a \emph{minimal forbidden minor} of \X\ if $G$ is not in \X\ but every proper minor of $G$ is in \X. Let \XX\ be the set of minimal forbidden minors of \X. By the graph minor theorem of \citet{RS-GraphMinors}, \XX\ is a finite set. For various minor-closed classes the list of minimal forbidden minors is known. Most famously, if \PP\ is the class of planar graphs, then the Kuratowski-Wagner theorem states that $\PPP=\{K_5,K_{3,3}\}$.
However, in general, determining the minimal forbidden minors for a particular minor-closed class is a challenging problem.



Let $\delta(G)$ be the minimum degree of a graph $G$. Let $\D{k}$ be the class of graphs $G$ such that every minor of $G$ has minimum degree at most $k$. Then \D{k} is minor-closed. Let \DD{k} be the set of minimal forbidden minors of \D{k}. By the graph minor theorem, \DD{k} is finite for each $k$. The structure of graphs in \DD{k} is the focus of this paper. For small values of $k$, it is known that $\DD{0}=\{K_2\}$ and $\DD{1}=\{K_3\}$ and $\DD{2}=\{K_4\}$ and $\DD{3}=\{K_5,K_{2,2,2}\}$; see \secref{Basics}. Determining \DD{4} is an open problem.

The majority this paper studies the case of general $k$ rather than focusing on small values. Our first main result shows that, in some sense, there are many graphs in \DD{k}. In particular, every sufficiently small $(k+1)$-regular graph is in \DD{k}. This result is proved in \secref{SmallRegularGraphs}.

\begin{theorem}
\thmlabel{SmallRegular}
Every $(k+1)$-regular graph with less than $\frac{4}{3}(k+2)$ vertices is in \DD{k}. Moreover, for all $k\equiv1\pmod{3}$ there is a $(k+1)$-regular graph on $\frac{4}{3}(k+2)$ vertices that is not in \DD{k}.
\end{theorem}

Our second main result characterises the graphs in \DD{k+1} that can be obtained from a graph in \DD{k} by adding one new vertex. 

\begin{theorem}
\thmlabel{AddVertex}
Let $S$ be a set of vertices in a graph $G\in\D{k}$. Let $G'$ be the graph obtained from $G$ by adding one new vertex adjacent to every vertex in $S$. Then $G'\in\DD{k+1}$ if and only if $S$ is the set of vertices of degree $k+1$ in $G$. 
\end{theorem}

\thmref{AddVertex} is proved in \secref{Construction} along with various corollaries of 
\twothmref{SmallRegular}{AddVertex}.

It is natural to expect that graphs in \DD{k} are, in some sense, highly connected. For example for $k\leq3$ all the graphs in \DD{k} are $(k+1)$-connected. However, this is not true in general. In \secref{Basics} we exhibit a graph in \DD{4} with connectivity $1$. In fact, our third main result, proved in \secref{BlockStructure}, constructs graphs in \DD{k} ($k\geq9$) with arbitrary block structure.

\begin{theorem}
Let $T$ be the block decomposition tree of some graph. 
Then for some $k$, $T$ is the block decomposition tree of some graph in \DD{k}.
\end{theorem}

A complete characterisation of graphs in \DD{k} is probably hopeless. So it is reasonable to restrict our attention to particular subsets of \DD{k}. A graph is \emph{complete $c$-partite} if the vertices can be $c$-coloured so that two vertices are adjacent if and only if they have distinct colours. Let $K_{n_1,n_2,\dots,n_c}$ be the complete $c$-partite graph with $n_i$ vertices in the $i$-th colour class. Since every graph in \DD{k} for $k\leq 3$ is complete multipartite, it is natural to consider the complete multipartite graphs in \DD{k}. Our fourth main result characterises the complete multipartite graphs in \DD{k}. 

\begin{theorem}
\thmlabel{CompleteMultipartite}
For all $k\geq1$, a complete multipartite graph $G$ is in \DD{k} if and only if for some $b\geq a\geq1$ and $p\geq2$, $$G=K_{a,\underbrace{b,\dots,b}_p}\enspace,$$
such that $k+1=a+(p-1)b$ and if $p=2$ then $a=b$.
\end{theorem}

\thmref{CompleteMultipartite} is proved in \secref{CompleteMultipartite}. 
Moreover, we prove that the same characterisation holds for the minimal forbidden complete multipartite minors for the class of graphs for which every minor has connectivity at most $k$. And \thmref{CMG-Treewidth} is an analogous result for graphs of treewidth at most $k$ and pathwidth at most $k$.

\section{Basics and Small Values of $k$}\seclabel{Basics}

This section gives some basic results about \DD{k} and reviews what is known about \DD{k} for small values of $k$. We have the following characterisation of graphs in \DD{k}. 

\begin{lemma}
\lemlabel{BasicDegree}
$G\in\DD{k}$ if and only if
\begin{enumerate}
\item[{\rm (D1)}] $\delta(G)=k+1$,
\item[{\rm (D2)}] every proper contraction minor of $G$ has minimum degree at most $k$,
\item[{\rm (D3)}] $G$ is connected,  and
\item[{\rm (D4)}] no two vertices both with degree at least $k+2$  are adjacent in $G$.
\end{enumerate}
\end{lemma}

\begin{proof}
$(\Longrightarrow)$ Suppose that $G\in\DD{k}$. That is, $\delta(G)\geq k+1$ and every minor of $G$ has minimum degree at most $k$. In particular, every 
contraction minor of $G$ has minimum degree at most $k$, thus proving (D2). 
If $G$ is not connected then each component of $G$ is a proper minor with minimum degree $k+1$. This contradiction proves (D3). If adjacent vertices $v$ and $w$ both have degree at least $k+2$, then $G-vw$ is a proper minor of $G$ with minimum degree at least $k+1$. This contradiction proves (D4). In particular, some vertex has degree $k+1$. Thus $\delta(G)=k+1$ and (D1) holds.

$(\Longleftarrow)$ Suppose that conditions (D1)--(D4) hold. Suppose on the contrary that some proper minor of $G$ has minimum degree at least $k+1$. Let $H$ be such a minor with the maximum number of edges. Since $G$ is connected, $H$ can be obtained by edge contractions and edge deletions only. (Deleting a non-isolated vertex $v$ can be simulated by contracting one and deleting the other edges incident to $v$.)\ Condition (D4) implies that every edge has an endpoint with degree $k+1$, implying that every proper subgraph of $G$ has minimum degree at most $k$. Hence at least one edge of $G$ was contracted in the construction of $H$. Since $H$ was chosen with the maximum number of edges, no edges were deleted in the construction of $H$. That is, $H$ is a contraction minor. Condition (D2) implies that $H$ has minimum degree at most $k$. This contradiction proves that every proper minor of $G$ has minimum degree at most $k$.
Thus condition (D1) implies that $G\in\DD{k}$.
\end{proof}



Observe that \lemref{BasicDegree} immediately implies that for all $k\geq0$,
\begin{equation} 
\eqnlabel{Complete}
K_{k+2}\in\DD{k}\enspace.
\end{equation}

Now consider small values of $k$. Observe that \D{0} is the class of edgeless graphs, and $\DD{0}=\{K_2\}$. Similarly \D{1} is the class of forests, and $\DD{1}=\{K_3\}$. Graphs in \D{2} are often called \emph{series-parallel}. \DD{2} and \DD{3} are easily determined; see \figref{DD23}.

\Figure{DD23}{\includegraphics{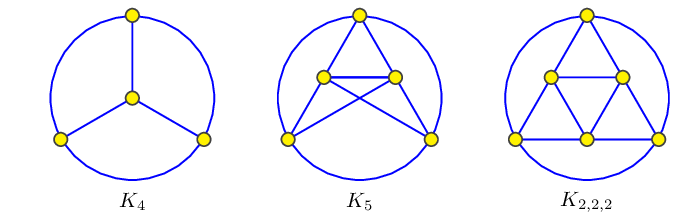}}{Graphs in $\DD{2}$ and $\DD{3}$.}

\begin{proposition}
$\DD{2}=\{K_4\}$.
\end{proposition}

\begin{proof}
By \eqnref{Complete}, $K_4\in\DD{2}$. Consider $G\in\DD{2}$. By \lemref{BasicDegree}, $G$ has minimum degree $3$. \citet{Dirac52} proved that every graph with minimum degree at least $3$ contains a $K_4$-minor; also see \citep{Tutte-NAW61, Hadwiger43, Zeidl58, Woodall-JGT92}.Thus $G$ contains a $K_4$-minor. If $G\not\cong K_4$, then the $K_4$-minor in $G$ is not proper, implying $G\not\in\DD{2}$ by \lemref{BasicDegree}. Hence $G\cong K_4$.
\end{proof}

\begin{proposition}
\proplabel{DDthree} 
$\DD{3}=\{K_5,K_{2,2,2}\}$.
\end{proposition}

\begin{proof}
By \eqnref{Complete}, $K_5\in\DD{3}$. Since $K_{2,2,2}$ is planar, every proper minor of $K_{2,2,2}$ is a planar graph on at most five vertices, which by Euler's Formula, has a vertex of degree at most $3$. Thus $K_{2,2,2}\in\DD{3}$. 

Consider $G\in\DD{3}$. By \lemref{BasicDegree}, $G$ has minimum degree $4$. 
In \appref{DegreeFour} we prove that every graph with minimum degree at least $4$ contains a $4$-connected minor\footnote{This result was attributed by \citet{Maharry-JGT99} to \citet{HJ-MA63}. While the authors acknowledge their less than perfect understanding of German, Halin and Jung actually proved that every $4$-connected graph contains $K_5$ or $K_{2,2,2}$ as a minor. This is confirmed by Tutte's review of the Halin and Jung paper in MathSciNet.}. \citet{HJ-MA63} proved that every $4$-connected graph contains $K_5$ or $K_{2,2,2}$ as a minor. Thus $G$ contains $K_5$ or $K_{2,2,2}$ as a minor. Suppose on the contrary that $G$ is isomorphic to neither $K_5$ nor $K_{2,2,2}$. Then $G$ contains $K_5$ or $K_{2,2,2}$ as a proper minor. Thus $G$ contains a proper minor with minimum degree 4, implying $G\not\in\DD{4}$ by \lemref{BasicDegree}. Hence $G$ is isomorphic to $K_5$ or $K_{2,2,2}$.
\end{proof}

Determining \DD{4} is an open problem. But we do know nine graphs in \DD{4}, as illustrated in \figref{MinDeg5Graphs}. One of these graphs is the \emph{icosahedron} $I$, which is the unique $5$-regular planar triangulation (on twelve vertices). \citet{Mader68} proved that every planar graph with minimum degree $5$ contains $I$ as a minor. More generally, \citet{Mader68} proved that every graph with minimum degree at least $5$ contains a minor in $\{K_6,I,C_5*\overline{K_3},K_{2,2,2,1}-e\}$, where $e$ is an edge incident to the degree-$6$ vertex in $K_{2,2,2,1}$. However, since $K_{2,2,2,1}-e$ has a degree-$4$ vertex, it is not in \DD{4}. \citet{Fijavz-PhD} proved that every graph on at most $9$ vertices with minimum degree at least $5$ contracts to $K_6$, $K_{2,2,2,1}$ or $C_5*\overline{K_3}$. The graphs $G_1$ and $G_2$ are  discussed further in \secref{GeneralSetting}. The graphs $D_1$ and $D_3$ are due to \citet{Fijavz-PhD}, while $D_2$ is due to \citet{Mader68}. Note that  $D_1$, $D_2$ and $D_3$ are not $5$-connected. In fact, $D_3$ has a cut-vertex. It is an example of a  more general construction given in \secref{BlockStructure}. In the language used there, $D_3$ is obtained from two copies of the single-horned graph $G_{5,4}$ by identifying the two horns. 

\Figure{MinDeg5Graphs}{\includegraphics{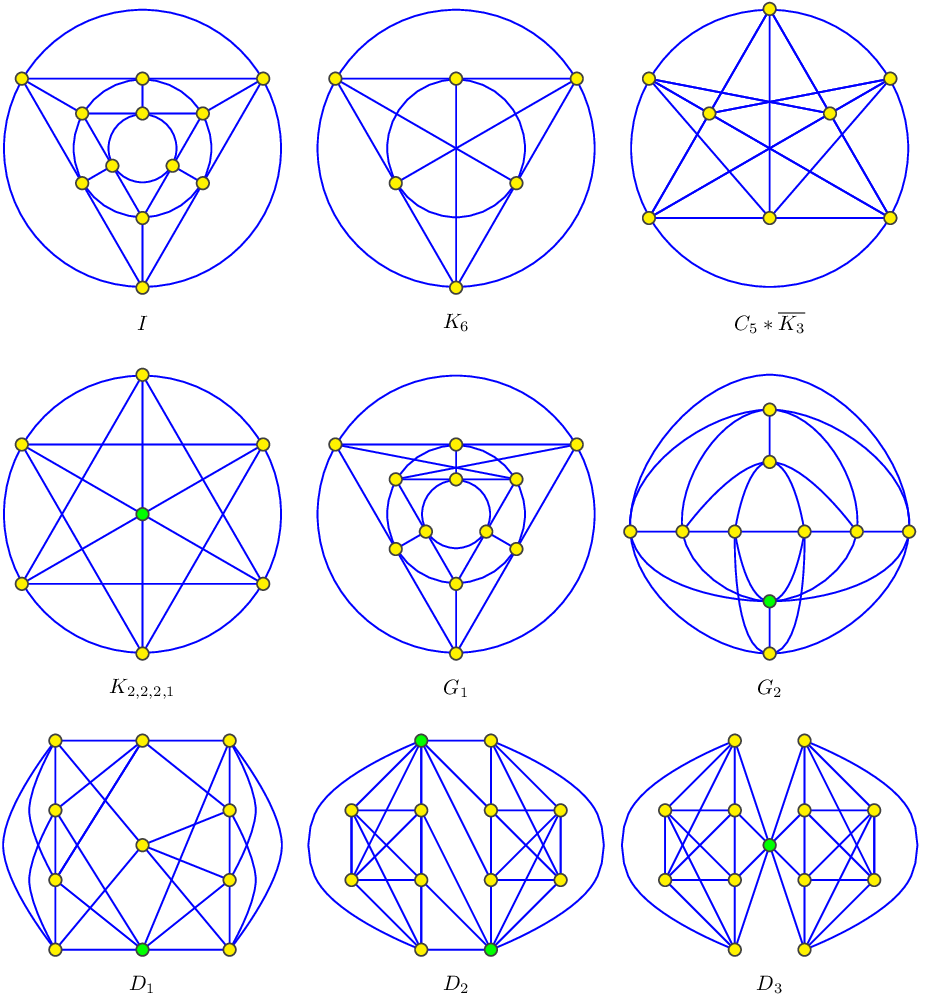}}{The known graphs in \DD{4}; vertices with degree more than $5$ are highlighted.}

\begin{proposition} 
$\{K_6,I,C_5*\overline{K_3},K_{1,2,2,2},G_1,G_2,D_1,D_2,D_3\}\subseteq\DD{4}$.
\end{proposition}

\begin{proof}
This result was verified by computer. 
(The code is available from the authors upon request.)\ 
We now give manual proofs for some of these graphs.

$K_6\in\DD{4}$ by \eqnref{Complete}.  

$I$ is not in \D{4} since it is $5$-regular. Every proper minor of $I$ is planar with at most eleven vertices. By Euler's Formula, every such graph has minimum degree at most $4$, and is thus in \D{4}. Hence $I\in\DD{4}$.

We now prove that $C_5*\overline{K_3}\in\D{4}$. Since $C_5*\overline{K_3}$ is $5$-regular, conditions (D1), (D3) and (D4) in \lemref{BasicDegree}. 
Suppose that $C_5*\overline{K_3}$ contains a proper contraction minor $H$ with $\delta(H)\geq 5$. Thus $|V(H)|\geq6$, and $H$ was obtained by at most two edge contractions.
Since every edge of $C_5*\overline{K_3}$ is in a triangle with a degree-5 vertex, $H$ was obtained by exactly two edge contractions.
Since each edge in the $C_5$ part of $C_5*\overline{K_3}$ is in three triangles, no edge in the $C_5$ part was contracted. 
Thus one contracted edge was $vw$ where $v\in C_5$ and $w\in\overline{K_3}$. Observe that $vw$ is in two triangles $vwx$ and $vwy$, where $x$ and $y$ are the neighbours of $v$ in $C_5$. Since both $x$ and $y$ have degree $4$ in $G/vw$, some edge incident to $x$ and some edge incident to $y$ is contracted in $H$. 
This is impossible since $x$ and $y$ are not adjacent, and only one contraction besides $vw$ is allowed. This contradiction proves that every proper contraction minor of $G$ has minimum degree at most $4$. Thus condition (D2) holds for $C_5*\overline{K_3}$, and $C_5*\overline{K_3}\in\DD{4}$.

That $K_{1,2,2,2}$ is in \DD{4} follows from \thmref{CMG-Degree} with $a=1$ and $b=2$ and $p=3$.

We now prove that $D_3\in\D{4}$. Observe that conditions (D1), (D3) and (D4) in \lemref{BasicDegree} hold for $D_3$. Suppose that $D_3$ contains a proper contraction minor $H$ with $\delta(H)\geq 5$. Thus $H=D_3/S$ for some $S\subseteq E(G)$ such that $|V(H)|=13-|S|$. Let $v$ be the cut-vertex in $D_3$. Let $G_1$ and $G_2$ be the subgraphs of $D_3$ such that $D_3=G_1\cup G_2$ where $V(G_1)\cap V(G_2)=\{v\}$. Let $S_i:=S\cap E(G_i)$. We have $|S_i|\leq|V(G_i)|-1=6$. Every edge of $D_3$ is in a triangle with a vertex distinct from $v$. Thus, if $|S_i|=1$ then some vertex in $H$ has degree at most $4$, which is a contradiction. If $2\leq |S_i|\leq 5$ then $G_i/S$ has at least two and at most five vertices, and every vertex in $G_i/S$ (except possibly $v$) has degree at most $4$, which is a contradiction. Thus $|S_i|\in\{0,6\}$. Now $|S_1|+|S_2|=|S|\leq 7$, as otherwise $H$ has at most five vertices. Thus $|S_1|=0$ and $|S_2|=6$ without loss of generality. Hence $H\cong G_1$, in which $v$ has degree $4$, which is a contradiction. Thus condition (D2) holds for $D_3$. Hence $D_3\in\DD{4}$.
\end{proof}

\section{A General Setting}\seclabel{GeneralSetting}

The following general approach for studying minor-closed class was introduced by \citet{Fijavz-PhD}. A \emph{graph parameter} is a function $f$ that assigns a non-negative integer $f(G)$ to every graph $G$, such that for every integer $k$ there is some graph $G$ for which $f(G)\geq k$. Examples of graph parameters include  minimum degree $\delta$,  maximum degree $\Delta$, (vertex-)connectivity $\kappa$,  edge-connectivity $\lambda$, chromatic number $\chi$, clique number $\omega$, independence number $\alpha$, treewidth $\textsf{tw}$, and pathwidth $\textsf{pw}$; see \citep{Diestel00} for definitions.

For a graph parameter $f$ and a graph $G$, let $\down{f}(G)$ be the maximum of $f(H)$ taken over all minors $H$ of $G$. Then \down{f} also is a graph parameter\footnote{Let $\up{f}(G)$ be the minimum of $f(H)$ where $G$ is a minor of $H$. Then the class of graphs $G$ with $\up{f}(G)\leq k$ is minor-closed, and we can ask the same questions for $\up{f}$ as for $\down{f}$. For example, the minor crossing number \citep{BFM-SJDM06, BCSV-ENDM, BFW} fits into this framework.}. For example, $\down{\omega}(G)$ is the order of the largest clique minor in $G$, often called the \emph{Hadwiger number} of $G$. 
Let $$\X_{f,k}:=\{G:\down{f}(G)\leq k\}\enspace.$$ 
That is, $\X_{f,k}$ is the class of graphs $G$ such that $f(H)\leq k$ for every minor $H$ of $G$. Then $\X_{f,k}$ is minor-closed, and the set $\XX_{f,k}$ of minimal forbidden minors is finite.

We have the following characterisation of graphs in $\XX_{f,k}$, analogous to \lemref{BasicDegree}.

\begin{lemma}
\lemlabel{Basic}
$G\in\XX_{f,k}$ if and only if $f(G)\geq k+1$ and every proper minor $H$ of $G$ has $f(H)\leq k$.
\end{lemma}

\begin{proof}
By definition, $G\in\XX_{f,k}$ if and only if $G\not\in\X_{f,k}$ but every proper minor of $G$ is in $\X_{f,k}$. That is, there exists a minor $H$ of $G$ with $f(H)\geq k+1$, but every proper minor $H$ of $G$ has $f(H)\leq k$. Thus the only minor $H$ of $G$ with $f(H)\geq k+1$ is $G$ itself. 
\end{proof}

\begin{lemma}
\lemlabel{alphabeta}
Let $\alpha$ and $\beta$ be graph parameters such that $\alpha(G)\leq\beta(G)$ for every graph $G$. Then $\X_{\beta,k}\subseteq\X_{\alpha,k}$ and $\{G:G\in\XX_{\beta,k},\alpha(G)\geq k+1\}\subseteq\XX_{\alpha,k}$.
\end{lemma}

\begin{proof}
For the first claim, let $G$ be a graph in $\X_{\beta,k}$. Then $\beta(H)\leq k$ for every minor $H$ of $G$. By assumption, $\alpha(H)\leq\beta(H)\leq k$. Hence $G\in\X_{\alpha,k}$, implying $\X_{\beta,k}\subseteq\X_{\alpha,k}$.

For the second claim, suppose that $G\in\XX_{\beta,k}$ and $\alpha(G)\geq k+1$. By \lemref{Basic} applied to $\beta$, $\beta(G)\geq k+1$ and every proper minor $H$ of $G$ has $\beta(H)\leq k$. By assumption, $\alpha(H)\leq\beta(H)\leq k$. Since $\alpha(G)\geq k+1$, \lemref{Basic} applied to $\alpha$ implies that $G\in \XX_{\alpha,k}$.
\end{proof}

Recall that $\delta$ and $\kappa$ are the graph parameters minimum degree and connectivity. Observe that $\D{k}=\X_{\delta,k}$. Let $$\C{k}:=\X_{\kappa,k}$$ be the class of graphs for which every minor has connectivity at most $k$. For $k\leq 3$, we have $\C{k}=\D{k}$ and $\CC{k}=\DD{k}$. That is, $\CC{1}=\{K_3\}$, $\CC{2}=\{K_4\}$, and $\CC{3}=\{K_5,K_{2,2,2}\}$. Determining \CC{4} is an open problem; \citet{Fijavz-PhD} conjectured that $\CC{4}=\{K_6,I,C_5*\overline{K_3},K_{1,2,2,2},G_1,G_2\}$.

\citet{Dirac-PLMS63} proved that every $5$-connected planar graph contains the icosahedron as a minor (which, as mentioned earlier, was generalised by \citet{Mader68} for planar graphs of minimum degree $5$). Thus the icosahedron is the only planar graph in \CC{4}. \citet{Fijavz-5} determined the projective-planar graphs in \CC{4} to be $\{K_6, I, G_1,G_2\}$. \citet{Fijavz-JCTB} determined the toroidal graphs in \CC{5} to be $\{K_7, K_{2,2,2,2}, K_{3,3,3}, K_9-C_9\}$. See \citep{Fijavz-EuJC04, FM-Comb03} for related results. Also relevant is the large body of literature on contractibility; see the surveys \citep{Kriesell-GC02, Mader-DM05}.


Let $$\T{k}:=\{G:\tw{G}\leq k\}\quad\text{ and }\W{k}:=\{G:\tw{G}\leq k\}$$ respectively be the classes of graphs with treewidth and pathwidth at most $k$. Since treewidth and pathwidth are minor-closed, $\T{k}=\X_{\textsf{tw},k}$ and $\W{k}=\X_{\textsf{pw},k}$. We have $$\kappa(G)\leq\delta(G)\leq\tw{G}\leq\pw{G}$$ for every graph $G$; see \citep{Bodlaender-TCS98,Diestel00}. Thus \lemref{alphabeta} implies that 
\begin{equation*}
\W{k}\subseteq\T{k}\subseteq\D{k}\subseteq\C{k},
\end{equation*}
and
\begin{align}
\eqnlabel{DDinCC}&\{G:G\in\DD{k},\kappa(G)\geq k+1\}\subseteq\CC{k}\\
\eqnlabel{TTinDD}&\{G:G\in\TT{k},\delta(G)\geq k+1\}\subseteq\DD{k}\\
\eqnlabel{WWinTT}&\{G:G\in\WW{k},\tw{G}\geq k+1\}\subseteq\TT{k}.
\end{align}
Thus the $(k+1)$-connected graphs that we show are in \DD{k} are also in \CC{k}. In particular, \thmref{SmallRegular} implies: 

\begin{theorem}
Every $(k+1)$-connected $(k+1)$-regular graph with less than $\frac{4}{3}(k+2)$ vertices is in \CC{k}.
\end{theorem}

The relationship between \CC{k} and \DD{k} is an interesting open problem.

\begin{open}
Is $\CC{k}\subseteq\DD{k}$ for all $k$? Is $\CC{k}=\{G:G\in\DD{k},\kappa(G)=k+1\}$ for all $k$?
\end{open}

Note that $\DD{4}\neq \CC{4}$ since there are graphs in \DD{4} with connectivity $1$; see \secref{BlockStructure}.

\section{General Values of $k$}

Let $G$ be a graph. A vertex of $G$ is \emph{low-degree} if its degree equals the minimum degree of $G$. A vertex of $G$ is \emph{high-degree} if its degree is greater than the minimum degree of $G$. Recall that every graph in \DD{k} has minimum degree $k+1$. Thus a vertex of degree $k+1$ in a graph in \DD{k} is low-degree; every other vertex is high-degree. \lemref{BasicDegree} implies that for every graph $G\in\DD{k}$, the high-degree vertices in $G$ form an independent set.




\begin{proposition}
\proplabel{ManyLows}
Every graph $G\in\DD{k}$ has at least $k+2$ low-degree vertices (of degree $k+1$).
\end{proposition}

\begin{proof}
Suppose on the contrary that $G$ has at most $k+1$ low-degree vertices. By \lemref{BasicDegree}, each high-degree vertex is only adjacent to low-degree vertices. Since a high-degree vertex has degree at least $k+2$, there are no high-degree vertices. Thus $G$ has at most $k+1$ vertices. Thus $G$ has maximum degree at most $k$, which is a contradiction.
\end{proof}

For a set $S$ of vertices in a graph $G$, a \emph{common neighbour} of $S$ is a vertex in $V(G)-S$ that is adjacent to at least two vertices in $S$. A \emph{common neighbour} of an edge $vw$ is a common neighbour of $\{v,w\}$. Common neighbours are important because of the following observation.

\begin{observation}
\obslabel{DegreeChange}
Let $vw$ be an edge of a graph $G$ with $p$ common neighbours. Let $H$ be the graph obtained from $G$ by contracting $vw$ into a new vertex $x$. Then 
$$\deg_H(x)=\deg_G(v)+\deg_G(w)-p-2.$$
For every common neighbour $y$ of $vw$,
$$\deg_H(y)=\deg_G(y)-1.$$
For every other vertex $z$ of $H$,
$$\deg_H(z)=\deg_G(z).$$
\end{observation}

\begin{proposition} 
\proplabel{CommonNeighbour}
For every graph $G\in\DD{k}$, every edge $vw$ of $G$ has a low-degree common neighbour.
\end{proposition}

\begin{proof}
If $k=1$ then $G=K_3$ and the result is trivial. Now assume that $k\geq 2$. 

Suppose on the contrary that for some edge $vw$ of $G$, every common neighbour of $vw$ (if any) is high-degree. By \lemref{BasicDegree}, at least one of $v$ and $w$ is low-degree (with degree $k+1$). Thus $v$ and $w$ have at most $k$ common neighbours. Let $u_1,\dots,u_p$ be the common neighbours of $v$ and $w$, where $0\leq p\leq k$. 

Let $H$ be the graph obtained from $G$ by contracting $vw$ into a new vertex $x$. The degree of each vertex of $G$ is unchanged in $H$, except for $v$, $w$, and each $u_i$. Since $\deg_G(u_i)\geq k+2$, we have $\deg_H(u_i)\geq k+1$. By \obsref{DegreeChange},
\begin{equation*}
\deg_H(x)=\deg_G(v)+\deg_G(w)-p-2
	\geq2(k+1)-p-2
	=2k-p\enspace.
\end{equation*}
Thus if $p\leq k-1$ then $\deg_H(x)\geq k+1$ and $H$ is a proper minor of $G$ with minimum degree at least $k+1$, implying $G\not\in\DD{k}$. 

Otherwise $p=k$, implying both $v$ and $w$ are low-degree vertices whose only neighbours are each other and the high-degree vertices $u_1,\dots,u_k$. Let $J$ be the graph obtained from $G$ by contracting $v,w,u_1$ into a new vertex $y$. Since each neighbour of $v$ is high-degree and each neighbour of $w$ is high-degree, if a vertex (other than $v,w,u_1$) is adjacent to at least two of $v,w,u_1$ then it is high-degree. Since no two high-degree vertices are adjacent, the only vertices  (other than $v,w,u_1$) that are adjacent to at least two of $v,w,u_1$ are $u_2,\dots,u_k$. Thus every vertex in $J$ (possibly except $y$) has degree at least $k+1$. Now $u_1$ has at least $k$ neighbours in $G$ outside of $\{v,w,u_2,\dots,u_k\}$. Thus $\deg_J(y)\geq k+(k-1)\geq k+1$, and $J$ is a proper minor of $G$ with minimum degree at least $k+1$, implying $G\not\in\DD{k}$.
\end{proof}

The next result says that for graphs in \DD{k}, every sufficiently sparse connected induced subgraph has a common neighbour.

\begin{proposition} 
For every graph $G\in\DD{k}$, for every connected induced subgraph $H$ of $G$ with $n$ vertices and $m\leq\half(k+1)(n-1)$ edges, there exists a vertex $x$ in $G-H$ adjacent to at least $\deg_G(x)-k+1\geq2$ vertices in $H$.
\end{proposition}

\begin{proof}
Suppose that for some connected induced subgraph $H$ with $n$ vertices and $m\leq\half(k+1)(n-1)$ edges, every vertex $x$ in $G-H$ is adjacent to at most $\deg_G(x)-k$ vertices in $H$. Let $G'$ be the graph obtained from $G$ by contracting $H$ into a single vertex $v$. The degree of every vertex in $G-H$ is at least $\deg_G(x)-(\deg_G(x)-k)+1=k+1$ in $G'$. Since $G$ has minimum degree $k+1$, we have
$$\deg_{G'}(v)
=\bracket{\sum_{w\in V(H)}\!\!\!\deg_G(w)}-2m
\geq n(k+1)-2m
\geq k+1.$$
Thus $G'$ is a proper minor of $G$ with minimum degree at least $k+1$. Hence $G\not\in\DD{k}$. This contradiction proves the result.
\end{proof}

\begin{corollary} 
For every graph $G\in\DD{k}$, for every clique $C$ of $G$ with at most $k+1$ vertices, there exists a vertex in $V(G)-C$ adjacent to at least two vertices of $C$.
\end{corollary}

\section{Small Regular Graphs are in \DD{k}}
\seclabel{SmallRegularGraphs}

In this section we show that that every $(k+1)$-regular graph with sufficiently few vertices is in \DD{k}. Moreover, the bound on the number of vertices is tight.

\begin{lemma}
\lemlabel{ManyTrianglesGraphs}
Let $G$ be a connected $(k+1)$-regular graph on $n$ vertices. 
If every edge of $G$ is in at least $2n-2k-5$ triangles, then $G\in\DD{k}$.
\end{lemma}

\begin{proof}
By assumption, conditions (D1), (D3) and (D4) of \lemref{BasicDegree} are satisfied by $G$. Suppose on the contrary that $H$ is a contraction minor of $G$ with minimum degree at least $k+1$. Let $S$ be the set of vertices of $G$ that are incident to an edge contracted in the construction of $H$. Let $vw$ be one such edge. We have $|S|\leq2(n-|V(H)|)\leq2n-2k-4$. By assumption, there is a set $T$ of vertices of $G$ that are adjacent to both $v$ and $w$, and $|T|\geq 2n-2k-5\geq|S|-1$. Thus there is at least one vertex $x\in T-(S-\{v,w\})$, which is a vertex of $H$. Since $x$ is adjacent to both endpoints of the contracted edge $vw$, $\deg_H(x)\leq k$. This contradiction proves condition (D2) for $G$. \lemref{BasicDegree} implies that $G\in\DD{k}$.
\end{proof}

\begin{lemma}
\lemlabel{ManyTriangles}
For every $(k+1)$-regular graph $G$ on $n$ vertices, every edge $vw$ of $G$ is in at least $2k+2-n$ triangles.
\end{lemma}

\begin{proof}
Say $vw$ is in $t$ triangles. Thus $v$ and $w$ have $t$ common neighbours. Thus $v$ has $k-t$ neighbours not adjacent to $w$, and $w$ has $k-t$ neighbours not adjacent to $v$. Thus $n\geq 2+t+2(k-t)=2k+2-t$, implying $t\geq2k+2-n$. 
\end{proof}

\begin{theorem}
\thmlabel{RegularGraphs}
Every $(k+1)$-regular graph $G$ with $n<\frac{4}{3}(k+2)$ vertices is in \DD{k}.
\end{theorem}

\begin{proof} 
Every disconnected $(k+1)$-regular graph has at least $2k+4$ vertices. Since $n<2k+4$ we can assume that $G$ is connected. By \lemref{ManyTriangles}, every edge of $G$ is in at least $2k+2-n$ triangles. Now $2k+2-n\geq 2n-2k-5$ since $n\leq\frac{1}{3}(4k+7)$. Thus  every edge of $G$ is in at least $2n-2k-5$ triangles. By \lemref{ManyTrianglesGraphs}, $G\in\DD{k}$. 
\end{proof}





\thmref{RegularGraphs} is best possible in the following sense.

\begin{proposition}
\proplabel{TightRegularGraphs}
For all $k\equiv1\pmod{3}$ there is a $(k+1)$-regular graph $G$ on $n=\frac{4}{3}(k+2)$ vertices that is not in \DD{k}.
\end{proposition}

\begin{proof}
Let $p:=\frac{1}{3}(k+2)$. Then $p\in\Z$. Let $G$ be the graph whose complement $\overline{G}$ is the disjoint union of $K_{p,p}$ and $K_{p,p}$. Then $G$ has $4p=n$ vertices, and every vertex has degree $n-1-p=k+1$. Observe that $G$ contains a matching $M$ of $p$ edges (between the two $K_{p,p}$ subgraphs in $\overline{G}$), such that every vertex is adjacent to at least one endpoint of every edge in $M$. Contracting each edge in $M$ we obtain a $K_{3p}$-minor in $G$, which has minimum degree $k+1$. Thus $G\not\in\DD{k}$.
\end{proof}


\thmref{RegularGraphs} can be rewritten in terms of complements.

\begin{corollary}
If $G$ is an $r$-regular graph on $n\geq4r+1$ vertices, then $\overline{G}\in\DD{n-r-2}$.
\end{corollary}

\section{A Construction}
\seclabel{Construction}

We now describe how a graph in \DD{k+1} can be constructed from a graph in \DD{k}. Let $G^+$ be the graph obtained from a graph $G$ by adding one new vertex that is adjacent to each vertex of minimum degree in $G$. If $G\in\DD{k}$ then the vertices of minimum degree are the low-degree vertices.

\begin{lemma}
\lemlabel{Construction}
If $G\in\DD{k}$ then $G^+\in\DD{k+1}$.
\end{lemma}

\begin{proof}
Let $v$ be the vertex of $G^+-G$. Every low-degree vertex in $G$ has degree $k+1$, and thus has degree $k+2$ in $G^+$. Every high-degree vertex in $G$ has degree at least $k+2$, which is unchanged in $G^+$. By \propref{ManyLows}, $G$ has at least $k+2$ low-degree vertices. Thus $v$ has degree at least $k+2$ in $G^+$. Thus $G^+$ has minimum degree $k+2$. Suppose on the contrary that $G\in\DD{k}$ but $G^+\not\in\DD{k+1}$. Thus there is a proper minor $H$ of $G^+$ with minimum degree at least $k+2$. If $v$ is not in a branch set of $H$, then $H$ is a minor of $G$, implying $H$ has minimum degree at most $k+1$, which is a contradiction. Now assume that $v$ is in some branch set $B$ of $H$. (Think of $B$ simultaneously as a vertex  of $H$ and as a set of vertices of $G^+$.)\ Now $H-B$ is a minor of $G$. If $H-B$ is $G$, then $B=\{v\}$ and $H$ is not a proper minor of $G^+$. Thus $H-B$ is a proper minor of $G$. Since $G\in\DD{k}$, $H-B$ has a vertex $X$ of degree at most $k$. Thus $X$ has degree at most $k+1$ in $H$, which is a contradiction. 
\end{proof}

We also have a converse result.

\begin{lemma}
\lemlabel{AddVertex}
Let $S$ be a set of vertices in a graph $G\in\D{k}$. Let $G'$ be the graph obtained from $G$ by adding one new vertex $v$ adjacent to every vertex in $S$. If $G'\in\DD{k+1}$ then $S$ is the set of low-degree vertices in $G$. 
\end{lemma}

\begin{proof}
Suppose that $G'\in\DD{k+1}$. If some low-degree vertex $x$ of $G$ is not in $S$, then $\deg_{G'}(x)=k+1$ and $G'\not\in\DD{k+1}$. Now assume that every low-degree vertex of $G$ is in $S$. Suppose on the contrary that some high-degree vertex $y$ of $G$ is in $S$.
Thus $\deg_G(y)\geq k+2$, implying $\deg_{G'}(y)\geq k+3$. By \propref{ManyLows} there are at least $k+2$ low-degree vertices of $G$, all of which are adjacent to $v$ in $G'$. Thus $\deg_{G'}(v)\geq k+3$. Hence $v$ and $y$ are adjacent vertices of degree at least $k+3$ in $G'$. Therefore $G'\not\in\DD{k+1}$ by \lemref{BasicDegree}. This contradiction proves that no high-degree vertex of $G$ is in $S$. Therefore $S$ is the set of low-degree degree vertices.
\end{proof}


Observe that \twolemref{Construction}{AddVertex} together prove \thmref{AddVertex}. \lemref{Construction} generalises as follows. For a non-negative integer $p$, let $G^{+p}$ be the graph obtained from a graph $G$ by adding $p$ independent vertices, each adjacent to every vertex in $G$. 

\begin{lemma}
\lemlabel{GeneralConstruction}
Let $G$ be a $(k+1)$-regular $n$-vertex graph in \DD{k}.
Then $G^{+p}\in\DD{k+p}$ whenever $0\leq p\leq n-k-1$.
\end{lemma}

\begin{proof}
Every vertex of $G$ has degree $k+1+i$ in $G^{+i}$.
Every vertex of $G^{+i}-G$ has degree $n$ in $G^{+i}$.
Thus, if $n>k+1+i$ then the vertices of minimum degree in $G^{+i}$ are exactly the vertices of $G$. Thus $G^{+i}=(G^{+(i-1)})^+$ whenever $1\leq i\leq n-k-1$. 
By induction on $i$, applying \lemref{Construction} at each step, we conclude that $G^{+i}\in\DD{k+i}$ and $G^{+p}\in\DD{k+p}$.
\end{proof}

\thmref{RegularGraphs} and \lemref{GeneralConstruction} imply:

\begin{corollary}
\corlabel{RegularGraphs}
Let $G$ be a $(k+1)$-regular graph with $n<\frac{4}{3}(k+2)$ vertices.
Then $G^{+p}\in\DD{k+p}$ whenever $0\leq p\leq n-k-1$.
\end{corollary}

\corref{RegularGraphs} implies:
 
\begin{lemma}
\lemlabel{Vida}
Let $L(G)$ denote the set of minimum degree vertices in a graph $G$. 
Let $p:=|G-L(G)|$. Suppose that \begin{itemize}
\item the minimum degree of $G$ is $k+1$, and
\item $|L(G)|<\frac{4}{3}(k+2-p)$, and
\item $V(G)-L(G)$ is an independent set of $G$, and 
\item every vertex in $V(G)-L(G)$ dominates $L(G)$.
\end{itemize}
Then $G\in\DD{k}$.
\end{lemma}

\begin{proof}
Let $X$ be the subgraph of $G$ induced by the vertices of degree $k+1$.
Thus $X$ is $(r+1)$-regular, where $r=k-p$.
Say $X$ has $n$ vertices. By assumption, 
$n<\frac{4}{3}(k+2-p)=\frac{4}{3}(r+2)$.
The high-degree vertices of $G$ have degree $n$, and the low-degree vertices of $G$ have degree $r+1+p$. Thus $n>r+1+p$. That is, $p<n-r-1$.
Thus, by \corref{RegularGraphs}, we have $G=X^{+p}\in\DD{r+p}=\DD{k}$.
\end{proof}


\section{Block Structure}
\seclabel{BlockStructure}


In this section we show that graphs in \DD{k} can have an arbitrary block decomposition tree\footnote{Let $G$ be a connected graph. Let $B$ denote the set of blocks of $G$ (that is, cut-edges and maximal $2$-connected components). Let $C$ denote the set of cut-vertices of $G$. The \emph{block decomposition tree of $G$} is the tree $T$ where $V(T)=B \cup C$, and $bc \in E(T)$ whenever the block $b$ contains $c$. A \emph{block decomposition tree} is a tree that is isomorphic to a block decomposition tree of some graph. The \emph{bipartition} of a tree $T$ is the partition of $V(T)$ obtained from a proper $2$-colouring of $T$. Since every cut-vertex is contained in at least two blocks, every leaf of a block decomposition tree $T$ belongs to the same bipartition class of $T$. Conversely, if a tree $T$ admits a bipartition of its vertices such that all leaves lie in the same bipartition class, then $T$ is a block decomposition tree.}. \twothmref{main}{main2} are the main results.
Note that every graph in \DD{k} has no cut-edge (except $K_2$), since a cut-edge can be contracted without decreasing the minimum degree.


A \emph{low-high tree} is a tree $T$ that admits a bipartition  $V(T)=\vl \cup \vh$, such that every vertex in \vl\ has degree at most $2$, and every vertex in \vh\ has degree at least $2$. Vertices in $\vl$ are called \emph{low}, and vertices in $\vh$ are called \emph{high}. Since every leaf in a low-high tree is low, every low-high tree is a block decomposition tree. 

In the following discussion, let $T$ be a low-high tree. 
Let $L$ be the set of leaves in $T$.
Let $r$ be an arbitrary high vertex of $T$, called the \emph{root}. 
For each edge $vw \in E(T)$, let $\dist(r,vw):=\min\{\dist(r,v),\dist(r,w)\}$. 
Let $B$ be the set of edges of $T$ at even distance from $r$.
Call these edges \emph{blue}.
Similarly let $R:=E(T)-B$ be the set of \emph{red} edges in $T$.
Since $r$ is high and each leaf is low, each leaf is at odd distance from $r$. Thus each edge incident with a leaf is blue.

\begin{lemma}
\lemlabel{p1}
The number of blue edges $|B|$ and the number of red edges $|R|$ do not depend on the choice of $r$.
\end{lemma}

\begin{proof}
It is enough to show that $|B|$ and $|R|$ do not change if we choose an alternative root $r'$ at distance $2$ from $r$. Let $R'$ and $B'$ be the sets of red and blue edges with respect to $r'$. Let $x$ be the common neighbour of $r$ and $r'$. Thus $rx \in B- B'$ and $r'x \in B' - B$. Apart from these edges, $B$ and $B'$ do not differ. Hence $|B|=|B'|$, and also $|R|=|R'|$.
\end{proof}

Define
\begin{equation*}
d:= 4|L|+2|R|\enspace.
\end{equation*}
Since $T$ has at least two leaves, $d\geq8$.
By \lemref{p1}, $d$ does not depend on the choice of $r$. 

For each edge $e=vw$ of $T$ such that $\dist(r,v)=\dist(r,w)-1$, let $T_e$ be the maximal rooted subtree of $T$ containing $vw$, and no other neighbour of $v$.

Define the function $\varphi: E(T) \rightarrow \N$ as follows.
For each blue edge $e$ in $T$, define
\begin{equation}
\eqnlabel{F1F2}
\varphi(e):=4|L\cap E(T_e)|+2|R\cap E(T_e)|\enspace.
\end{equation}
Now consider a red edge $vw$ in $T$ with $\dist(r,v)=\dist(r,w)-1$. Thus $\dist(r,v)$ is odd, $v$ is low, and $\deg(v)=2$. Let $uv$ be the blue edge incident to $v$. Define 
\begin{equation}
\eqnlabel{F3}
\varphi(vw):= d+2-\varphi(uv)
\enspace.
\end{equation}

\Figure{Tree}{\includegraphics{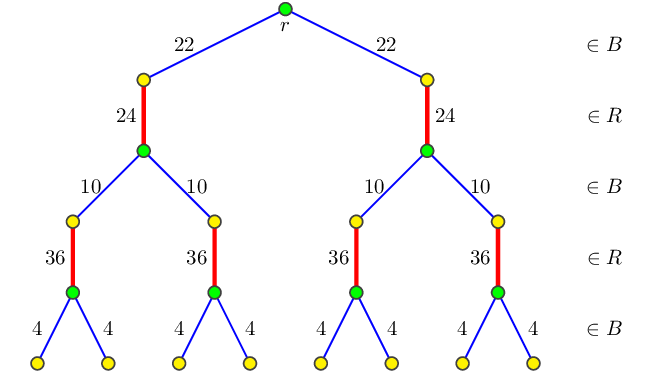}}{An example of the edge labelling $\varphi$ with $|R|=6$ and $|B|=14$ and $|L|=8$ and $d=2\cdot 6+4\cdot 8=44$. Red edges are drawn thick.}

The next lemma immediately follows from \eqnref{F3}.

\begin{lemma}
\lemlabel{p3}
If $v$ is a low vertex of degree $2$ and $v$ is incident with
edges $e$ and $f$, then $\varphi(e)+\varphi(f)=d+2$.
\end{lemma}

The sum of $\varphi$ values around a high vertex is also constant.

\begin{lemma}
\lemlabel{p4}
Let $v$ be a high vertex and let $E_v$ be the set of edges incident with $v$. Then $\sum_{e \in E_v} \varphi(e)=d$.
\end{lemma}

\begin{proof}
First suppose that $v=r$. Then 
\begin{align*}
\sum_{rx\in E_r}\varphi(rx)
\;=\;\sum_{rx\in E_r}4|L\cap E(T_{rx})|+2|R\cap E(T_{rx})|
\;=\;4|L\cap E(T)|+2|R\cap E(T)|
\;=\;d\enspace.
\end{align*}
Now assume that $v\neq r$. Since $v$ is high, $\dist(v,r)$ is even, and $v$ is incident to one red edge $uv$ (where $u$ is the neighbour of $v$ closer to $r$ than $v$). Thus $u$ is low, and $\deg(u)=2$. Let $t$ be the other neighbour of $u$. 
Let $e_1,\dots,e_k$ be the blue edges incident to $v$.
Then 
\begin{align*}
\sum_{e \in E_v} \varphi(e)
\;=\;&\varphi(uv)+\sum_{i=1}^k\varphi(e_i)\\
\;=\;&d+2-\varphi(tu)+\sum_{i=1}^k\varphi(e_i)\\
\;=\;&d+2-4|L\cap E(T_{tu})|-2|R\cap E(T_{tu})|
+\sum_{i=1}^k4|L\cap E(T_{e_i})|+2|R\cap E(T_{e_i})|\enspace.
\end{align*}
Observe that $L\cap E(T_{tu})=\bigcup_iL\cap E(T_{e_i})$
and $R\cap E(T_{tu})-\bigcup_iR\cap E(T_{e_i})=\{uv\}$.
Thus
$$\sum_{e \in E_v} \varphi(e)\;=\;d+2-2=d\enspace.$$
\end{proof}

Observe that, in principle, the definition of $\varphi$ depends on the choice of $r$. However, this is not the case.

\begin{lemma}
\lemlabel{p5}
Let $r$ and $r'$ be high vertices of $T$, and let $\varphi$ and $\varphi'$ be the functions defined above using $r$ and $r'$ as roots, respectively. Then $\varphi=\varphi'$.
\end{lemma}

\begin{proof}
Since $T$ is connected, it is enough to show that $\varphi=\varphi'$ whenever $\dist(r,r')=2$. 
Let $x$ be the common neighbour of $r$ and $r'$.
Let $B'$ be the set of blue edges with respect to $r'$.
Now $B$ and $B'$ (as well as $R$ and $R'$) differ only in $rx$ and $r'x$.
Since \eqnref{F1F2} only considers $\varphi$ and $\varphi'$  values of blue edges away from the root, $\varphi(e)=\varphi'(e)$ for each $e \in B\cap B'$.
Since each edge incident with $r$ or $r'$ apart from $rx$ and $r'x$ is in $B \cap B'$, and since $d$ is invariant, \eqnref{F3} shows that $\varphi$ and $\varphi'$ match on every edge in $R \cap R'$. Finally \lemref{p4} implies that $\varphi$ and $\varphi'$ also match on edges between $rx$ and $r'x$.
\end{proof}

\begin{lemma}
\lemlabel{p6}
$\varphi(e) \ge 4$ for every edge $e \in E(T)$.
\end{lemma}

\begin{proof}
While the colour of an edge $e$ may depend on the choice of $r$, \lemref{p5} says that $\varphi(e)$ does not depend on the choice of $r$. Every edge can be made blue for an appropriate choice of $r$, and $\varphi(e)\geq4$ for every blue edge $e$ by \eqnref{F1F2}. 
\end{proof}

And now for something completely different. 
Let $e=u_1u_2$ and $f=u_3u_4$ be two independent edges in the complete graph
$K_{d+1}$, where $d \ge 4$. The \emph{single-horned graph} $G_{d,4}$ is obtained from $K_{d+1}$ by adding a new vertex $x$, connecting $x$ to $u_1,u_2,u_3,$ and $u_4$ and
removing edges $e$ and $f$. 
Observe that $\deg(x)=4$. Call $x$ the \emph{horn} of $G_{d,4}$. Call 
the remaining vertices the \emph{original vertices} of $G_{d,4}$, which
all have degree $d$.

Let $a,b \ge 4$ be even integers such that $d=a+b-2$.
Choose matchings $M_a$ and $M_b$ with $\frac{a}{2}$ and $\frac{b}{2}$ edges, respectively, that cover all the vertices of $K_{d+1}$. 
Hence $M_a$ and $M_b$ share exactly one vertex. 
Take two new vertices $x_a$ and $x_b$ and join $x_a$ to every vertex of $M_a$ and $x_b$ to every vertex of $M_b$. Next delete the edges of $M_a$ and $M_b$.
The resulting graph is called the \emph{double-horned graph} $G_{d,a,b}$. As above, $x_a$ and $x_b$ are called the \emph{horns} of $G_{d,a,b}$, and the remaining vertices, all of degree $d$, are the \emph{original vertices}.

Let $e=uv$ be an edge in a single- or double-horned graph $G$.
If $u$ or $v$ is a horn in $G$, then the vertex $uv$ is a \emph{horn} in $G/e$ and is \emph{original} otherwise. Inductively, we can define horns and original
vertices for every contraction minor of a horned graph.

\begin{lemma}
\lemlabel{contr}
Let $G'$ be a proper contraction minor of a horned graph $G_{d,4}$ or $G_{d,a,b}$. If $G'$ contains an original vertex, then some original vertex of $G'$ has degree less than $d$.
\end{lemma}

\begin{proof}
We shall leave the proof for $G_{d,4}$ to the keen reader.
Let $G$ be the doubly horned graph $G_{d,a,b}$, and let $F \subseteq E(G)$ such 
that $G/F =G'$. If $|F| \ge 3$, then $G'$ has at most $d$ vertices,
and \emph{all} its vertices have degree less than $d$. Now assume that $|F| \le 2$.

Let $e=uv$ be an edge connecting a pair of original vertices. 
There are at least $7 \le a+b-1$ original vertices in $G$ and at least
three original vertices are connected with both $u$ and $v$. 
Thus $G/e$ has at least three original vertices of degree less than $d$, which cannot all be eliminated by a single additional contraction.
Hence every edge in $F$ is incident with a horn. Let $x$ be a horn incident with
$e$. At least two neighbours of $x$ (which are original vertices)
have degree less than $d$ in $G/e$, yet by the above argument, the edge between them cannot be contracted.
\end{proof}

We are now ready to state the first theorem of this section.

\begin{theorem}
\thmlabel{main}
For every low-high tree $T$, there is an integer $d$ and a graph $G$ such that:
\begin{enumerate}
\item[{\rm (G1)}] $G$ is $d$-regular,
\item[{\rm (G2)}] $T$ is the block decomposition tree of $G$,
\item[{\rm (G3)}] $8\leq d \le 4 |E(T)|$, and
\item[{\rm (G4)}] $G \in \DD{d+1}$.
\end{enumerate}
\end{theorem}

\begin{proof}
Adopt the above notation. Let $d:= 4|L| + 2|R|$. By construction, $d\geq8$. 
Note that $d \le 4 |E(T)|$ with equality only if $T$ is a star.

For every leaf $u$ of $T$, let $G_u$ be a copy of the single-horned graph $G_{d,4}$.
For every non-leaf low vertex $v$ of $T$ incident with edges $e$ and $f$, 
let $G_v$ be a copy of the double-horned graph $G_{d,a,b}$, where $a:=\varphi(e)$ and $b:=\varphi(f)$. Note that $a,b\geq4$ by \lemref{p6}.

Observe that there is a natural correspondence between the set of horns in the above graphs and their degrees, and between $E(T)$ and their $\varphi$ values.
As illustrated in \figref{Example}, identifying horns wherever the edges in $T$ have a common (high) end-vertex gives rise to a $d$-regular graph $G$ (by \lemref{p4}). Hence $G$ satisfies (G1), (G2) and (G3).

\Figure{Example}{\includegraphics[width=\textwidth]{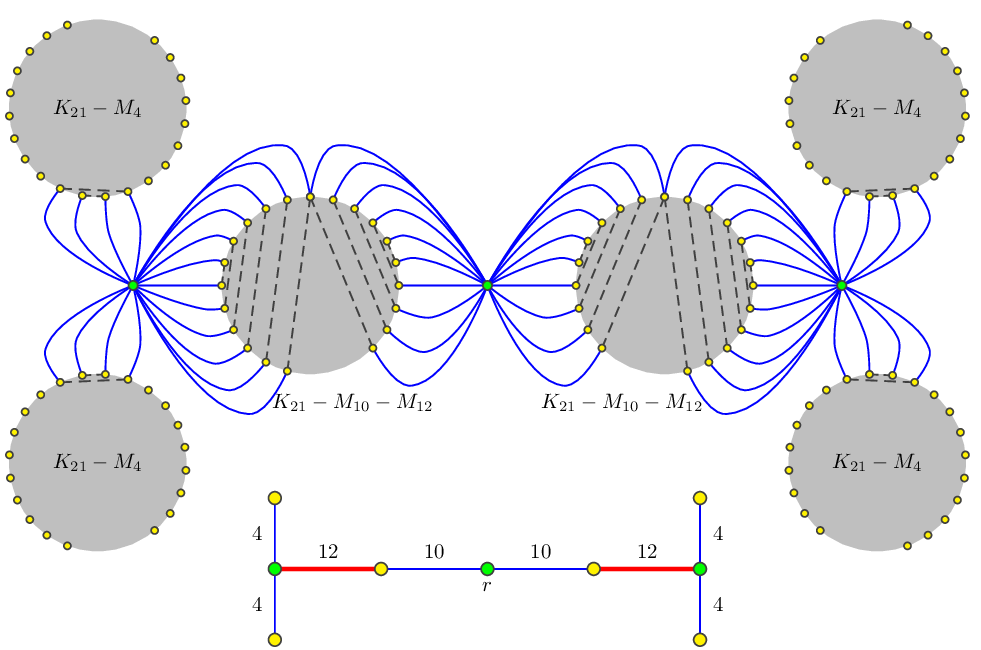}}{The graph $G$ produced from the given low-high tree with $d=4\cdot 4+2\cdot2=20$. Shaded regions represent cliques minus the dashed matchings.}

Since $G$ is connected and $d$-regular, \lemref{BasicDegree} implies that to establish (G4) it suffices to show that every proper contraction minor of $G$ has a vertex of degree less than $d$. Suppose on the contrary that there is a proper contraction minor $G'=G/E'$ of $G$ with $\delta(G) \ge d$. Take such a $G'$ with the minimum number of vertices. Thus $G'$ has no cut-edges, since contracting a cut-edge does not decrease $\delta$ (since $G'\not\cong K_2$).

Let $H$ be an arbitrary block of $G$ and consider $H/E'$.
Suppose that $H/E'$ is not contracted to a single vertex.
Now $H/E'\not\cong K_2$ (as this would either be a
nonexistent cut-edge in $G'$ or would imply that $G'$ has a vertex of degree 1
which is also absurd). 
But if $H/E'$ has at least three vertices and $H/E'$ is a proper minor
of $H$, then by \lemref{contr}, $H/E'$ has an inner vertex of degree less than $d$. 
Hence $H/E'$ is either trivial or is left intact in a contraction.

So we may assume that $G'$ is obtained by shrinking several blocks of
$G$ to single vertices. We may assume that $G'$ is obtained by first
contracting $k_i \ge 0$ inner blocks of $G$, and later contracting
$k_e \ge 0$ end-blocks of $G$, where $k_i+k_e \ge 1$.
Let $G^*$ be the graph obtained after contracting the inner blocks.

Now $k_i > 0$, as otherwise $G'$ is a proper subgraph of $G$.
By shrinking $k_i$ inner blocks we have reduced the number of
cut-vertices by $k_i$, and also reduced the sum of their
degrees by $k_i(d+2)$; see~\twolemref{p3}{p4}. Hence $G^*$ has at least one
\emph{cut-vertex} $v$ of degree less than $d$, and since $G' \ne G^*$, at least one
contraction of an end-block follows.
Finally, contracting an end-block cannot increase $\deg(v)$.
This contradiction completes the proof of (G4).
\end{proof}

We now prove that minor-minimal minimum-degree graphs can have arbitrary block structure.

\begin{theorem}
\thmlabel{main2}
For every block decomposition tree $T$, 
there is an integer $d$ and a graph $G$ such that
\begin{enumerate}
\item[{\rm (H1)}] $T$ is the block decomposition tree of $G$,
\item[{\rm (H2)}] $\delta(G) \le 8 |E(T)|$, and
\item[{\rm (H3)}] $G \in \DD{d+1}$ where $d=\delta(G)$.
\end{enumerate}
\end{theorem}

\begin{proof}
Let $V_c \cup V_b$ be the bipartition of $V(T)$, such that every leaf of $T$ is in $V_b$. Let $H_b$ denote the set of vertices in $V_b$ with degree at least $3$ in $T$.
Thus $T$ is low-high if and only if $H_b=\emptyset$. 
By \thmref{main} we may assume that $T$ is not low-high, and $H_b\neq\emptyset$. Choose an arbitrary vertex $x \in H_b$.

Let $T'$ be the tree obtained from $T$ by subdividing each edge that is incident with a vertex in $H_b$ once. Each such subdivision vertex and each vertex in $V_b-H_b$ has degree at most $2$ in $T'$. Each vertex in $V_c \cup H_b$ has degree at least $2$ in $T'$. Thus $T'$ is low-high. In particular, $x$ is a high vertex of $T'$.


Now $|E(T')| < 2 |E(T)|$ since at least one edge of $T$ is incident with a leaf and did not get subdivided in the construction of $T'$. By \thmref{main} there exists an integer $d' \le 4 E(T') < 8 E(T)$ and a $d'$-regular graph $G' \in \DD{d'+1}$ such that $T'$ is the block decomposition tree of $G'$. In order to keep the arguments below as simple as possible, assume that $G'$ \emph{is} the graph obtained by the construction in the proof of \thmref{main}. Observe that every block of $G'$ contains at least $12$ vertices, since $T$ has at least one vertex in $H_b$. Note that the cut-vertices of $G'$ come in two flavours: ones that correspond to vertices of $V_c$, and ones that correspond to vertices of $H_b$. Similarly, every non-cut-vertex of $G'$ corresponds to a vertex of $V_b- H_b$.

Now define a partition of $V(G')$ into bags $\{B_y : y \in V_b\}$ labelled by vertices $V_b$, satisfying the following conditions:
\begin{enumerate}
\item[(C1)] for every $y \in H_b$ the bag $B_y$ contains the cut-vertex $c$ that corresponds to $y$ as well as the interior vertices of every block that contains $c$,
\item[(C2)] for every $y \in V_b- H_b$ the bag $B_y$ contains every interior vertex of a block  that corresponds to $y \in H_b$.
\end{enumerate}
We have so far partitioned every vertex of $G'$ that is not a cut-vertex corresponding to a vertex in $V_c$.
\begin{enumerate}
\item[(C3)] if $c$ is a cut-vertex corresponding to a vertex of $V_c$, then let $c_x$ be its neighbour on some shortest $c$--$x$ path in $G'$, and put $c$ in the bag $B$ that already contains $c_x$.
\end{enumerate}

Observe that every block of $G'$ contains $d'+1$ interior vertices, hence
every bag $B_y$ contains at least $d'+1$ vertices.

Finally we obtain $G$ from $G'$ by adding for each bag $B_y$ of $G'$ a new  vertex $\tilde{y}$ which is made adjacent to every vertex of its bag
$B_y$. Now $G'$ is a subgraph of $G$ and every $v \in V(G')$ has degree equal to $d'+1$, and new vertices have degree at least $d'+1$. Call this process \emph{bag extension} and let $d:=d'+1$.

Now $G$ contains two types of blocks: \emph{small blocks} that contain
interior vertices of exactly one block of $G'$, and \emph{big blocks} that contain interior vertices of several blocks of $G'$. Observe that
every big block $B$ contains a separating set of size two comprised of its new vertex and a vertex from $H_b$.

Let $B'$ (respectively, $B$) be an end-block of $G'$ ($G$) and let $c$ be a cut-vertex that separates $B'$ ($B$) from the rest of $G'$ ($G$).
By the construction of $G'$ there are exactly four edges incident with 
$c$ whose other end-vertex is in $B' (B)$.

Let $e$ be an arbitrary edge of $G$ that is not one of the four edges incident to some cut-vertex of an end-block. Assume that $e$ belongs to block $B$ of $G$. Then there are at least six vertices of degree $d$ in $G$ that are all adjacent to both end-vertices of $e$. This implies that $G/e$ contains at least six vertices of degree less than $d$, and no contraction of an additional two edges of $B$ can eliminate all the vertices of degree less than $d$.

First observe that an end-block of $G$ contains exactly $d+2$ vertices and the other small blocks contain exactly $d+3$ vertices. Every big block on the other hand contains a pair of vertices: the new vertex and a cut-vertex of $G'$ corresponding to a vertex in $H_b$.

It remains to prove that $G \in \DD{d+1}$. Since every edge has an end-vertex of degree $d$, no edge-deleted subgraph of $G$ has minimum degree at least $d$.

Hence we only have to consider contraction minors of $G$.
Let $F \subseteq E(G)$ be a nonempty edge set and let $G^*=G/F$.
We may split $F =F' \cup F^*$ so that $F' \subseteq E(G')$.

A block $B$ of $G$ may either get contracted to a single vertex, get partially contracted, or survive the contraction of $F$ without changes.

First assume that $B/F$ gets partially contracted. If $B'$ is an end-block, then $B/F$ has exactly $d+1$ vertices obtained by contracting a single edge. This is not possible as a vertex of degree less than $d$ would be created.
If $B$ is a small block, then contracting any edge of $B$ leaves at least six vertices of degree less than $d$ in $B$. Since $B$ has $d+3$ vertices in the beginning, an additional two contractions decrease the vertex count below $d+1$, which is absurd.

Let $B$ be a big block that gets partially contracted.
If contraction identifies the new vertex $n$ of $B$ and
a cut-vertex $c$ of $G'$ corresponding to a vertex in $H_b$ then
$B/{nc}$ contains at least six vertices of degree less than $d$ in \emph{every}
block $B'$ of $G'$ that is a subgraph of $B$.
Since $B'$ contains $d+2$ vertices, $B'/F$ is trivial for every
$B' \subseteq B$, which is nonsense.
Otherwise assume that $B' \subseteq B$ is a block of $G'$ that gets partially contracted. The $d+1$ interior vertices of $B'$ are separated from the rest of $G$ by three vertices. This implies that at most three edges are contracted in order to contract $B'$ partially. Yet a single contraction produces six vertices of degree less than $d$ in $B'$, so that an additional two contractions do not suffice.

Hence no block of $G$ gets partially contracted in $G/F$. Now $G/F$ may be obtained from $G'/F$ by extension of bags, where $G'/F$ is a contraction of $G'$ that either identifies a block of $G'$ or leaves it unchanged.
In this case, $G'/F$ contains a vertex of degree less than $d'$, and
bag extension can only increase its degree by one. This completes the proof of \thmref{main2}.
\end{proof}

\begin{open}
By the Robertson-Seymour graph minor theorem, every graph in \DD{k} has at most $f(k)$ vertices, for some function $f$. It would be interesting to obtain a simple proof of this result, and to obtain good bounds on $f$.\\
 By \thmref{main} with $T=K_{1,s}$, there is a graph $G\in\DD{4s+1}$ with $1+s(4s+1)$ vertices. Does every graph in \DD{k} have $O(k^2)$ vertices? \\
By \thmref{main} with $T=P_{2s+1}$, there is a graph $G\in\DD{2s+7}$ with diameter $2s$. Does every graph in \DD{k} have $O(k)$ diameter?
\end{open}


\section{Complete Multipartite Graphs}
\seclabel{CompleteMultipartite}

This section characterises the complete multipartite graphs in \DD{k}, in \CC{k}, in \TT{k}, and in \WW{k}. See \citep{Lucena-DAM07, Chleb-DAM02, Rama-SJDM97} for other results on treewidth obstructions. We first prove three lemmas about complete multipartite graphs. The first says that complete multipartite graphs are highly connected. 

\begin{lemma}
\lemlabel{CMGconnected}
Every complete multipartite graph $G$ with minimum degree $k$ is $k$-connected. Moreover, if $vw$ is an edge of $G$ such that both $v$ and $w$ have degree at least $k+1$, then $G-vw$ is $k$-connected. 
\end{lemma}

\begin{proof}
Let $x$ and $y$ be distinct vertices in $G$. It suffices to prove that there is a set of $k$ internally disjoint paths between $x$ and $y$ that avoid $vw$. Let $R$ be the set of vertices coloured differently from both $x$ and $y$. 

First suppose that $x$ and $y$ have the same colour. Then $\deg(x)=\deg(y)\geq k$, and $P:=\{xry:r\in R\}$ is a set of $\deg(x)$ internally disjoint paths between $x$ and $y$. If $vw$ is in some path in $P$, then without loss of generality $v=x$, implying $\deg(x)\geq k+1$, and at least $k$ paths in $P$ avoid $vw$.

Now assume that $x$ and $y$ are coloured differently. Let $S:=\{x_1,x_2\dots,x_p\}$ be the colour class that contains $x$, where $x=x_p$. Let $T:=\{y_1,y_2\dots,y_q\}$ be the colour class that contains $y$, where $y=y_q$. Without loss of generality, $n-p=\deg(x)\leq\deg(y)=n-q$, implying $q\leq p$. Thus $$P:=\{xy\}\cup\{xry:r\in R\}\cup\{xy_ix_iy:i\in[q-1]\}$$ is a set of $\deg(x)$ internally disjoint paths between $x$ and $y$. If $\deg(x)\geq k+1$ then at least $k$ paths in $P$ avoid $vw$. Now assume that $vw$ is in some path in $P$, but $\deg(x)=k$. Since each vertex $x_i$ has the same degree as $x$, and $v$ and $w$ both have degree at least $k+1$, the only possibility is that $v=y$ and $w=r$ for some $r\in R$ (or symmetrically $w=y$ and $v=r$). Thus $\deg(x)<\deg(y)$ and $q<p$. Replace the path $xry$ in $P$ by the path $xrx_{p-1}y$, which is internally disjoint from the other paths in $P$. 
\end{proof}

\begin{lemma}
\lemlabel{CMGequal}
Let $G$ be a complete multipartite graph on $n$ vertices. Then
$$\kappa(G)=\delta(G)=\tw{G}=\pw{G}=n-\alpha(G).$$
\end{lemma}

\begin{proof}
The degree of a vertex $v$ equals $n$ minus the size of the colour class that contains $v$. Since every independent set is contained within a colour class, the size of the largest colour class equals $\alpha(G)$. Thus $\delta(G)=n-\alpha(G)$. We have $\kappa(G)\leq\delta(G)\leq\tw{G}\leq\pw{G}$ for every graph $G$; see \citep{Bodlaender-TCS98,Diestel00}. By \lemref{CMGconnected}, $\kappa(G)\geq\delta(G)$. Thus it suffices to prove that $\delta(G)\geq\pw{G}$ for every complete multipartite graph $G$. Let $S=\{v_1,\dots,v_{\alpha(G)}\}$ be a largest colour class in $G$. Let $X:=V(G)-S$. Observe that $(X\cup\{v_1\},X\cup\{v_2\},\dots,X\cup\{v_p\})$
is a path decomposition of $G$ with width $|X|=n-\alpha(G)=\delta(G)$. Thus $\pw{G}\leq\delta(G)$.
\end{proof}

\begin{lemma}
\lemlabel{CMG}
If $H$ is a minor of a complete multipartite graph $G$, then $H$ can be obtained from $G$ by a sequence of edge contractions, vertex deletions, and edge deletions, such that each operation does not increase the minimum degree, connectivity, treewidth, or pathwidth.
\end{lemma}

\begin{proof}
Every minor of a graph can be obtained by a sequence of edge contractions and vertex deletions, followed by a sequence of edge deletions. Contracting an edge or deleting a vertex in a complete multipartite graph produces another complete multipartite graph. Edge deletions do not increase the minimum degree, connectivity, treewidth, or pathwidth. Thus by \lemref{CMGequal}, it suffices to prove that edge contractions and vertex deletions in complete multipartite graphs do not increase the minimum degree.

Say $G=K_{a_1,\dots,a_p}$ has $n$ vertices. Then $G$ has minimum degree $n-\max_ia_i$. Let $G'$ be the graph obtained from $G$ by contracting an edge. Then $G'$ is a complete multipartite graph $K_{1,a'_1,\dots,a'_p}$ with $n-1$ vertices, where $a_i-1\leq a'_i\leq a_i$. Thus $$\delta(G')=n-1-\max_ia'_i\leq n-1-\max_i(a_i-1)=n-\max_ia_i=\delta(G)\enspace.$$ Now let $G'$ be the graph obtained from $G$ by deleting a vertex. Then $G'$ is a complete multipartite graph $K_{a'_1,\dots,a'_p}$ with $n-1$ vertices, where $a_i-1\leq a'_i\leq a_i$. By the same argument as before, $\delta(G')\leq \delta(G)$. 
\end{proof}

We now state and prove our first characterisation.

\begin{theorem}
\thmlabel{CMG-Degree}
For all $k\geq1$, the following are equivalent for a complete multipartite graph $G$:\\
\hspace*{5mm}\textup{(a)} $G\in\CC{k}$\\
\hspace*{5mm}\textup{(b)} $G\in\DD{k}$\\
\hspace*{5mm}\textup{(c)} for some $b\geq a\geq1$ and $p\geq2$ such that $k+1=a+(p-1)b$, $$G=K_{a,\underbrace{b,\dots,b}_p}\enspace,$$
\hspace*{10mm}and if $p=2$ then $a=b$.
\end{theorem}

\begin{proof}
(b) $\Longrightarrow$ (a): Say $G\in\DD{k}$. 
By \lemref{BasicDegree}, $\delta(G)=k+1$. 
By \lemref{CMGequal}, $\kappa(G)=k+1$.
By \eqnref{DDinCC}, $G\in\CC{k}$.

\medskip (a) $\Longrightarrow$ (c): 
Consider a complete multipartite graph $G\in\CC{k}$. Thus $\kappa(G)\geq k+1$ by \lemref{Basic}. If $\kappa(G)\geq k+2$ then $\kappa(G-e)\geq k+1$ for any edge $e$ of $G$, implying $G\not\in\CC{k}$ by \lemref{Basic}. Now assume that $\kappa(G)=k+1$. Thus $\delta(G)=k+1$ by \lemref{CMGequal}. 

Suppose on the contrary that adjacent vertices $v$ and $w$ in $G$ both have degree at least $k+2$. By \lemref{CMGconnected}, $G-vw$ is $k$-connected, implying $G\not\in\CC{k}$. This contradiction proves that no two high-degree vertices in $G$ are adjacent. If two vertices in a complete multipartite graph have distinct degrees, then they are adjacent. Thus the high-degree vertices in $G$ have the same degree, and the vertices of $G$ have at most two distinct degrees. Since the degree of each vertex $v$ equals $|V(G)|$ minus the number of vertices in the colour class that contains $v$, the colour classes of $G$ have at most two distinct sizes. Hence for some $b\geq a\geq 1$ and $p,q\geq1$, $$G=K_{\underbrace{a,\dots,a}_q,\underbrace{b,\dots,b}_p}.$$

Hence $\kappa(G)=aq+b(p-1)=k+1$. If $a=b$ then, taking $q=1$, we are done. Now assume that $a<b$. Thus $q=1$ as otherwise two high-degree vertices are adjacent. Thus
$$G=K_{a,\underbrace{b,\dots,b}_p}\enspace.$$

Suppose on the contrary that $p=1$. Then $G=K_{a,b}$ and $\kappa(G)=a$. Contracting one edge in $G$ gives $K_{1,a-1,b-1}$, which by \lemref{CMGequal} also has connectivity $a$, implying $G\not\in\CC{k}$. This contradiction proves that $p\geq2$.

Now suppose that $p=2$. Then $G=K_{a,b,b}$ and $\kappa(G)=a+b$. Contracting one edge gives $K_{1,a,b-1,b-1}$, which by \lemref{CMGequal} also has connectivity $a+b$ (since $a<b$), implying $G\not\in\CC{k}$. This contradiction proves that if $p=2$ then $a=b$.


\medskip (c) $\Longrightarrow$ (b) Let $$G=K_{a,\underbrace{b,\dots,b}_p}\enspace,$$
for some $b\geq a\geq1$ and $p\geq2$, such that $k+1=a+(p-1)b$ and if $p=2$ then $a=b$. Thus $G$ has minimum degree $k+1$ by \lemref{CMGequal}. Suppose on the contrary that $G\not\in\DD{k}$. By \lemref{Basic}, $G$ has a proper minor $H$ with $\delta(H)\geq k+1$. By \lemref{CMG}, every minor of $G$ in the sequence from $G$ to $H$ has minimum degree at most $k+1$. Thus we can assume that $H$ was obtained from $G$ by a single edge contraction, a vertex deletion, or an edge deletion. In each case we prove that $\delta(H)\leq k$, which is the desired contradiction.

First suppose that $H$ is obtained from $G$ by an edge contraction. Then 
$$\text{(i) }H=K_{1,a-1,b-1\underbrace{b,\dots,b}_{p-1}}
\;\;\text{ or }\;\;
\text{(ii) }H=K_{1,a,b-1,b-1,\underbrace{b,\dots,b}_{p-2}}\enspace.$$

In case (i), $\delta(H)=1+(a-1)+(b-1)+(p-2)b=k$. 
In case (ii) with $p\geq3$, $\delta(H)=1+a+2(b-1)+(p-3)b=k$. 
Now consider case (ii) with $p=2$. By assumption, $a=b$. Thus $H=K_{1,a,a-1,a-1}$ has minimum degree $1+2(a-1)=k$.

Now suppose that $H$ is obtained from $G$ by a vertex deletion. Then 
$$\text{(i) }H=K_{a-1,\underbrace{b,\dots,b}_{p}}
\;\;\;\text{ or }\;\;\;
\text{(ii) }H=K_{a,b-1,\underbrace{b,\dots,b}_{p-1}}\enspace.$$
In case (i), $\delta(H)=(a-1)+(p-1)b=k$.
In case (ii), $\delta(H)=a+(b-1)+(p-2)b=k$ (since $p\geq2$).

In $G$, every edge is incident to a vertex of degree $k+1$. Thus, if $H$ is obtained from $G$ by an edge deletion, then $\delta(H)\leq k$. 
\end{proof}


The remainder of this section is devoted to characterising the complete multipartite graphs in \TT{k} and in \WW{k}. We start with a lemma about independent sets in complete multipartite graphs.

\begin{lemma}
\lemlabel{CMG-IndependentSets}
For every edge $vw$ in a complete multipartite graph $G$, every independent set in $G-vw$ is either $\{v,w\}$ or is also independent in $G$.  Thus if $\alpha(G)\geq2$ (that is, $G$ is not a complete graph) then $\alpha(G-vw)=\alpha(G)$.
\end{lemma}

\begin{proof}
Let $G':=G-vw$. Let $I$ be an independent set in $G'$ that is not independent in $G$. Thus both $v$ and $w$ are in $I$. Let $S$ be the colour class containing $v$. Every vertex not in $S\cup\{w\}$ is adjacent to $v$ in $G'$. Thus $I\subseteq S\cup\{w\}$. Every vertex in $S-\{v\}$ is adjacent to $w$ in $G'$. Thus $I:=\{v,w\}$. Hence every independent set in $G'$ is either $\{v,w\}$ or is also independent in $G$. Thus $\alpha(G')=\alpha(G)$ whenever $\alpha(G)\geq2$. 
\end{proof}

To prove lower bounds on treewidth we use the following idea. Let $G$ be a graph. Two subgraphs $X$ and $Y$ of $G$ \emph{touch} if $X\cap Y\neq\emptyset$ or there is an edge of $G$ between $X$ and $Y$. A \emph{bramble} in $G$ is a set of pairwise touching connected subgraphs. The subgraphs are called \emph{bramble elements}. A set $S$ of vertices in $G$ is a \emph{hitting set} of a bramble \B\ if $S$ intersects every element of \B. The \emph{order} of \B\ is the minimum size of a hitting set. The following `Treewidth Duality Theorem' shows the intimate relationship between treewidth and brambles.

\begin{theorem}[\citet{SeymourThomas-JCTB93}]
\thmlabel{TreewidthBramble}
A graph $G$ has treewidth at least $k$ if and only if $G$ contains a bramble of order at least $k+1$.
\end{theorem}

For example, say $G$ is a complete multipartite graph on $n$ vertices. Let $S$ be a set of vertices in $G$, one from each colour class; that is, $S$ is a maximum clique in $G$. Then it is easily seen that $\B:=E(G)\cup S$ is a bramble of order $n-\alpha(G)+1$, and thus $\tw{G}\geq n-\alpha(G)$ by \thmref{TreewidthBramble} (confirming \lemref{CMGequal}). The next two lemmas give circumstances when an edge can be deleted from a complete multipartite graph without decreasing the treewidth.

\begin{lemma}
\lemlabel{CMG-DeleteEdge}
Let $G$ be a complete multipartite graph with $\alpha(G)\geq3$, such that at least two colour classes contain at least two vertices. Let $vw$ be an edge, where both $v$ and $w$ are in colour classes that contain at least two vertices. Then $\tw{G-vw}=\tw{G}$.
\end{lemma}

\begin{proof}
Say $G$ has $n$ vertices. Let $G':=G-vw$. By \twolemref{CMGequal}{CMG-IndependentSets}, $\tw{G}=n-\alpha(G)=n-\alpha(G')$. Clearly $\tw{G'}\leq\tw{G}$. Thus it suffices to prove that $\tw{G'}\geq n-\alpha(G')$. 

Since $v$ and $w$ are in colour classes that contain at least two vertices, there is a set $S$ of vertices, such that both $v$ and $w$ are not in $S$,  and each colour class has exactly one vertex in $S$. Thus $S$ is a maximum clique in $G$ and in $G'$. Let $\B:=E(G')\cup S$. 

We now prove that \B\ is a bramble in $G'$. Each element of \B\ induces a connected subgraph in $G'$. Every pair of vertices in $S$ are adjacent. Say $x\in S$ and $pq\in E(G')$. Since $p$ and $q$ have distinct colours, $x$ is coloured differently from $p$ or $q$, and thus $x$ is adjacent to $p$ or $q$ (since $x\ne v$ and $x\ne w$). Hence $x$ touches $pq$. Say $pq\in E(G')$ and $rs\in E(G')$. If $\{p,q\}\cap\{r,s\}\neq\emptyset$ then $pq$ and $rs$ touch. So assume that $p,q,r,s$ are distinct. Thus there are at least two edges in $G$ between $\{p,q\}$ and $\{r,s\}$, one of which is not $vw$. 
Hence $pq$ touches $rs$. Therefore \B\ is a bramble in $G'$. 

Let $H$ be a minimum hitting set of \B. If  $|H|\geq n-\alpha(G')+1$, then \B\ has order at least $n-\alpha(G')+1$, implying $\tw{G'}\geq n-\alpha(G')$ by \thmref{TreewidthBramble}, and we are done. Now assume that $|H|\leq n-\alpha(G')$. 

Since every edge of $G'$ is in \B, $H$ is a vertex cover of $G'$, and $V(G')-H$ is an independent set of $G'$. Thus $n-|H|\leq\alpha(G')$. Hence $|H|=n-\alpha(G')$, and 
$V(G')-H$ is a maximum independent set of $G'$. By \lemref{CMG-IndependentSets}, every independent set of $G'$ is $\{v,w\}$ or is an independent set of $G$. 
Since $\alpha(G')\geq3$, $\{v,w\}$ is not a maximum independent set. Hence
$V(G)-H$ is a maximum independent set of $G$. That is, $V(G)-H$ is a colour class in $G$, which implies that $H$ does not contain one vertex in $S$, and $H$ is not a hitting set of \B. This is the desired contradiction. 
\end{proof}

\begin{lemma}
\lemlabel{CMG-DeleteSpecialEdge}
Let $G$ be a complete multipartite graph with $\alpha(G)\geq2$, and at least one singleton colour class. Let $vw$ be an edge, where $v$ is in a singleton colour class, and $w$ is in a colour classes that contains at least two vertices. Then $\tw{G-vw}=\tw{G}$.
\end{lemma}

\begin{proof}
Say $G$ has $n$ vertices. Let $G':=G-vw$. By \twolemref{CMGequal}{CMG-IndependentSets}, $\tw{G}=n-\alpha(G)=n-\alpha(G')$. Clearly $\tw{G'}\leq\tw{G}$. Thus it suffices to prove that $\tw{G'}\geq n-\alpha(G')$. 

By assumption, there is a set $S$ of vertices, such that $w\not\in S$,  and every colour class has exactly one vertex in $S$. Thus $v\in S$. Note that $S$ is a maximum clique in $G$ and in $G'$. Let $\B:=E(G')\cup S$. 

We now prove that \B\ is a bramble in $G'$. Each element of \B\ induces a connected subgraph in $G'$. Every pair of vertices in $S$ are adjacent. Consider $v\in S$ and $pq\in E(G')$. Since $v$ is in a singleton colour class, $v$ is adjacent to both $p$ and $q$ in $G$, and thus $v$ is adjacent to $p$ or $q$ in $G'$. Hence $v$ touches $pq$. Now consider $x\in S-\{v\}$ and $pq\in E(G')$. Since $p$ and $q$ have distinct colours, $x$ is coloured differently from $p$ or $q$, and thus $x$ is adjacent to $p$ or $q$ (since $x\ne v$ and $x\ne w$). Hence $x$ touches $pq$. Finally consider two edges $pq\in E(G')$ and $rs\in E(G')$. If $\{p,q\}\cap\{r,s\}\neq\emptyset$ then $pq$ and $rs$ touch. So assume that $p,q,r,s$ are distinct. Thus there are at least two edges in $G$ between $\{p,q\}$ and $\{r,s\}$, one of which is not $vw$. Hence $pq$ touches $rs$. Therefore \B\ is a bramble in $G'$. 

Let $H$ be a minimum hitting set of \B. If  $|H|\geq n-\alpha(G')+1$, then \B\ has order at least $n-\alpha(G')+1$, implying $\tw{G'}\geq n-\alpha(G')$ by \thmref{TreewidthBramble}, and we are done. Now assume that $|H|\leq n-\alpha(G')$. 

Since every edge of $G'$ is in \B, $H$ is a vertex cover of $G'$, and $V(G')-H$ is an independent set of $G'$. Thus $n-|H|\leq\alpha(G')$. Hence $|H|=n-\alpha(G')$, and $V(G')-H$ is a maximum independent set of $G'$. By \lemref{CMG-IndependentSets}, every independent set of $G'$ is $\{v,w\}$ or is an independent set of $G$. If $V(G')-H=\{v,w\}$ then $H$ does not contain $v$, and $H$ is not a hitting set of \B, which is a contradiction. Otherwise, $V(G)-H$ is a maximum independent set of $G$. That is, $V(G)-H$ is a colour class in $G$, which implies that $H$ does not contain some vertex in $S$, and $H$ is not a hitting set of \B. This is the desired contradiction. 
\end{proof}

\begin{theorem}
\thmlabel{CMG-Treewidth}
For all $k\geq1$, the following are equivalent for a complete multipartite graph $G$:\\
\hspace*{5mm}\textup{(a)} $G\in\TT{k}$\\
\hspace*{5mm}\textup{(b)} $G\in\WW{k}$\\
\hspace*{5mm}\textup{(c)} $G=K_{k+2}$, or $k\geq 3$ is odd and  $\displaystyle G=K_{\underbrace{2,\dots,2}_{(k+3)/2}}$.
\end{theorem}

\begin{proof}
(b) $\Longrightarrow$ (a): Say $G\in\WW{k}$. 
By \lemref{Basic}, $\pw{G}=k+1$. 
By \lemref{CMGequal}, $\tw{G}=k+1$.
By \eqnref{WWinTT}, $G\in\TT{k}$.

\medskip (a) $\Longrightarrow$ (c): Say $G\in\TT{k}$. 
By \lemref{Basic}, $\tw{G}\geq k+1$. 
If $\tw{G}\geq k+2$ then $\tw{G-v}\geq k+1$ for any vertex $v$ of $G$, implying $G\not\in\TT{k}$ by \lemref{Basic}. Now assume that $\tw{G}=k+1$. 
Thus $\delta(G)=k+1$ by \lemref{CMGequal}, and 
$G\in\DD{k}$ by \eqnref{TTinDD}. By \thmref{CMG-Degree}, $$G=K_{a,\underbrace{b,\dots,b}_p}\enspace,$$ 
for some $b\geq a\geq1$ and $p\geq2$, 
such that $k+1=a+(p-1)b$ and if $p=2$ then $a=b$.

\textbf{\boldmath Case.  $b=1$:} Then $a=1$ and $G=K_{k+2}$, and we are done. 

\textbf{\boldmath Case.  $b=2$:} Then $k+3=a+2p$.
If $a=1$, then by \lemref{CMG-DeleteSpecialEdge}, $\tw{G-e}=\tw{G}$ for some edge $e$ of $G$, implying $G\not\in\T{k}$ by \lemref{Basic}. 
Otherwise $a=2$. Thus $k=2p-1$ is odd, and $k\geq3$ since $p\geq2$. 
Hence 
$$\displaystyle G=K_{\underbrace{2,\dots,2}_{(k+3)/2}}\enspace.$$

\textbf{\boldmath Case.  $b\geq3$:} Then $\alpha(G)\geq3$. Since $p\geq2$, there are at least two colour class that contain at least two vertices, and by \lemref{CMG-DeleteEdge}, $\tw{G-e}=\tw{G}$ for some edge $e$ of $G$, implying $G\not\in\T{k}$ by \lemref{Basic}. 

\medskip (c) $\Longrightarrow$ (b): If $G=K_{k+2}$ then $G\in\WW{k}$ by \lemref{Basic}. Now assume that $k\geq 3$ is odd and 
$$\displaystyle G=K_{\underbrace{2,\dots,2}_{(k+3)/2}}\enspace.$$
Thus $\pw{G}=k+1$ by \lemref{CMGequal}. Suppose on the contrary that $G\not\in\WW{k}$. By \lemref{Basic}, $G$ has a proper minor $H$ with $\pw{H}\geq k+1$. By \lemref{CMG}, every minor of $G$ in the sequence from $G$ to $H$ has pathwidth at most $k+1$. Thus we can assume that $H$ was obtained from $G$ by a single edge contraction, a vertex deletion, or an edge deletion. Since an edge contraction or a vertex deletion produce another complete multipartite graph, and the minimum degree of a complete multipartite graph equals its pathwidth (\lemref{CMGequal}), the same proof used in \thmref{CMG-Degree} shows that $\pw{H}\leq k$. Now assume that $H=G-vw$ for some edge $vw$ of $G$. Let $x$ be the other vertex in the colour class that contains $v$. Let $y$ be the other vertex in the colour class that contains $w$. Let $S:=V(G)-\{v,w,x,y\}$. Then $(S\cup\{v,y\},S\cup\{x,y\},S\cup\{x,w\})$ is a path decomposition of $H$ with width $k$, which is the desired contradiction.
\end{proof}

\begin{open}
Complete multipartite graphs have diameter $2$. Are there generalisations of  \twothmref{CMG-Degree}{CMG-Treewidth} for all diameter-$2$ graphs in \DD{k} or in \TT{k}?
\end{open}

\section*{Acknowledgements} 

Thanks to Vida Dujmovi\'c who first proved \lemref{Vida}.


\def\soft#1{\leavevmode\setbox0=\hbox{h}\dimen7=\ht0\advance \dimen7
  by-1ex\relax\if t#1\relax\rlap{\raise.6\dimen7
  \hbox{\kern.3ex\char'47}}#1\relax\else\if T#1\relax
  \rlap{\raise.5\dimen7\hbox{\kern1.3ex\char'47}}#1\relax \else\if
  d#1\relax\rlap{\raise.5\dimen7\hbox{\kern.9ex \char'47}}#1\relax\else\if
  D#1\relax\rlap{\raise.5\dimen7 \hbox{\kern1.4ex\char'47}}#1\relax\else\if
  l#1\relax \rlap{\raise.5\dimen7\hbox{\kern.4ex\char'47}}#1\relax \else\if
  L#1\relax\rlap{\raise.5\dimen7\hbox{\kern.7ex
  \char'47}}#1\relax\else\message{accent \string\soft \space #1 not
  defined!}#1\relax\fi\fi\fi\fi\fi\fi} \def\Dbar{\leavevmode\lower.6ex\hbox to
  0pt{\hskip-.23ex \accent"16\hss}D}

\appendix
\section{Graphs with Minimum Degree Four}
\applabel{DegreeFour}

In this appendix we prove the following result.

\begin{theorem}
\thmlabel{DegreeFour}
Every graph with minimum degree at least $4$ contain a $4$-connected minor.
\end{theorem}

The following stronger result enables an inductive proof of \thmref{DegreeFour}.

\begin{lemma}
Let $G$ be a graph with at least $5$ vertices, such that the vertices of degree at most $3$ induce a clique. Then $G$ contains a $4$-connected minor.
\end{lemma}

\begin{proof}
Let $G$ be a counterexample with the minimum number of vertices.
Let $K=\{v_1,\dots,v_{|K|}\}$ denote the (possibly empty) clique of vertices of degree at most $3$.

In each case below we exhibit a proper minor $G'$ of $G$, for which it is easy to verify that the vertices of degree at most $3$ induce a clique. Moreover, $|V(G')|\geq5$ since there is a vertex of degree at least $4$ in $G$, whose degree does not decrease in $G'$. Thus $G'$ satisfies the conditions of the lemma, which contradicts the minimality of $G$.

If $e$ is an edge incident to a vertex of degree at most $2$, then let $G':=G/e$. Now assume that $\delta(G)\geq3$.

Let $S$ be a minimal separator in $G$, and let $\{G_1,G_2\}$ be the corresponding separation, so that $S=V(G_1 \cap G_2)$.
Without loss of generality, $K \subseteq G_1$. If $|S|=1$, then let $G':=G_2$.
If $|S|=2$, say $S=\{s_1,s_2\}$, then there exists an $s_1$--$s_2$ path in $G_1$, and so $G':= G_2 + s_1s_2$ is the desired minor.

Thus $G$ is $3$-connected and each vertex in $K$ has degree $3$. 
Let $N(K)$ denote the subgraph induced by the neighbours of $K$.

First suppose that $|K|=2$. Both $v_1$ and $v_2$ have at least two neighbours in $N(K)$. If $|N(K)| \ge 3$, then at most one vertex is adjacent to both $v_1$ and $v_2$. Let $G':=G/v_1v_2$. If, on the other hand, $N(K)=\{u_1,u_2\}$, then let $G'$ be obtained from $G$ by contracting the triangle $v_1v_2u_1$. Since $G$ has a vertex of degree at least $4$ other than $u_1,u_2$, so does $G'$.

If $|K|=1$ or $|K|=3$ then $|N(K)|=3$. Every vertex of $N(K)$ is adjacent to exactly one vertex of $K$, so $G$ has more than $|K|+|N(K)|$ vertices. If $N(K)$ induces a clique, then let $G':=G-V(K)$. Otherwise let $u_1$ be a vertex whose degree in $N(K)$ is as small as possible (is at most $1$) and let $v_1$ be its unique neighbour in $K$.
In this case, let $G':=G/u_1v_1$. Therefore we may assume that $G$ is $3$-connected with $\delta(G) \geq 4$. 

Suppose that $G$ contains a $3$-separation $\{G_1,G_2\}$ with separator  $S=\{s_1,s_2,s_3\}=V(G_1\cap G_2)$. Consider the subgraph $G_1$. Each vertex in $S$ has degree at least $1$ in $G_1$. Now $V(G_1-S)\neq\emptyset$ and every vertex in $G_1-S$ has degree at least $4$ in $G_1$. A forest that contains a vertex of degree at least $4$ has at least $4$ leaves. Thus $G_1$ is not a forest. Hence $G_1$ contains a cycle $C$. By Menger's Theorem, there are three disjoint $C$--$S$ paths in $G_1$. By contracting $C$ together with these three paths to a triangle on $S$, observe that $G':=G_2+s_1s_2+s_1s_3+s_2s_3$ is the desired minor. Hence $G$ has no $3$-separation and is thus $4$-connected.
\end{proof}

For completeness we include a proof of the following theorem of \citet{HJ-MA63} based on classical results by Wagner, Whitney and Tutte.

\begin{theorem}[\citep{HJ-MA63}]
\thmlabel{HJ}
Every $4$-connected graph contains $K_5$ or $K_{2,2,2}$ as a minor.
\end{theorem}

\begin{proof}
Suppose that $G$ is $4$-connected and has no $K_5$-minor. Thus $G$ is planar by Wagner's characterisation of graphs with no $K_5$-minor \citep{Wagner37}. Fix a plane embedding of $G$. Let $v$ be any vertex of $G$. Let $w_1,w_2,w_3,w_4$ be four of the neighbours of $v$ in cyclic order around $v$. Let $C$ be the facial cycle in the induced plane embedding of $G-v$, such that the interior of $C$ contains $v$. \citet{Whitney-AJM33c} proved that every $3$-connected planar graph has a unique plane embedding. Moreover, \citet{Tutte63} proved that the faces of this embedding are exactly the induced nonseparating cycles. Since $G-v$ is $3$-connected, each face in the induced plane embedding of $G-v$ is an induced nonseparating cycle. In particular, $C$ is induced and nonseparating in $G-v$. Since $C$ is separating, $(G-v)-C$ is connected. Since $C$ is induced, each vertex $w_i$ has  exactly two neighbours in $C$, and at least one neighbour in  $(G-v)-C$. Hence, contracting $(G-v)-C$ to a single vertex, and contracting $C$ to the $4$-cycle $(w_1,w_2,w_3,w_4)$ produces a $K_{2,2,2}$-minor in $G$.
\end{proof}

Note that \thmref{HJ} can also be concluded from a theorem of \citet{Maharry-JGT99}, who proved that every $4$-connected graph with no $K_{2,2,2}$ minor is isomorphic to the square of an odd cycle, which is easily seen to contain a $K_5$-minor. \twothmref{DegreeFour}{HJ} imply:

\begin{corollary}
Every graph with minimum degree at least $4$ contain $K_5$ or $K_{2,2,2}$ as a minor.
\end{corollary}


\end{document}